\documentclass[11pt]{article}

\usepackage{amsthm,amsfonts,amssymb,amsmath,color}
\usepackage{bm}
\usepackage{bbm}
\usepackage[utf8x]{inputenc}
\usepackage{enumerate}
\usepackage[english]{babel}
\usepackage{esint}
\usepackage[mathscr]{eucal}
\usepackage{hyperref}
\hypersetup{
	colorlinks = true, 
	urlcolor = blue, 
	linkcolor = blue, 
	citecolor = blue 
}

\usepackage{url}

\usepackage{layout}
\allowdisplaybreaks

\numberwithin{equation}{section}

\renewcommand\d{\partial}
\renewcommand\a{\alpha}
\renewcommand\b{\beta}

\newcommand\s{\sigma}

\newcommand\R{\mathbb R}

\newcommand\Z{\mathbb Z}

\newcommand{\weak}{\rightharpoonup}

\def\O{\Omega}

\def\l{\lambda}

\def\bvp{\bm{\varphi}}
\def\e{\varepsilon}

\newcommand\br{\begin{rem}}
	\newcommand\er{\end{rem}}
\newcommand\bp{\begin{pmatrix}}
	\newcommand\ep{\end{pmatrix}}
\newcommand\be{\begin{equation}}
	\newcommand\ee{\end{equation}}
\newcommand\ba{\begin{equation}\begin{aligned}}
		\newcommand\ea{\end{aligned}\end{equation}}

\newcommand\nn{\nonumber}

\usepackage{geometry}
\newcommand{\Id}{{\rm Id }}

\newcommand{\vq}{\textbf{\textup{q}}}
\newcommand{\vw}{\textbf{\textup{w}}}

\newcommand{\uu}{{\mathbf u}}
\newcommand{\vv}{{\mathbf v}}
\newcommand{\ff}{{\mathbf f}}

\newcommand{\ww}{{\mathbf w}}

\newcommand{\vr}{\varrho}

\newcommand{\vu}{\vc{u}}
\newcommand{\vf}{\vc{f}}

\newcommand{\vc}[1]{{\bf #1}}

\newcommand{\Grad}{\nabla_x}

\newcommand{\dx}{\, {\rm d} {x}}
\newcommand{\dt}{\, {\rm d} t }

\newcommand{\vU}{\vc{U}}

\newcommand{\dive}{{\rm div}_x}

\newtheorem{defi}{Definition}[section]
\newtheorem{theorem}[defi]{Theorem}
\newtheorem{proposition}[defi]{Proposition}
\newtheorem{lemma}[defi]{Lemma}

\newtheorem{remark}[defi]{Remark}

\begin{document}
	
	\title{Qualitative derivation of a density dependent incompressible Darcy law}
	
	\author{Danica Basari\'c\footnote{Politecnico di Milano, Department of Mathematics, Via E. Bonardi 9, 20133 Milano, Italy, danica.basaric@polimi.it} \and Florian Oschmann\footnote{Mathematical Institute of the Czech Academy of Sciences, \v{Z}itn\'a 609/25, 140 00 Praha, Czech republic, oschmann@math.cas.cz} \and Jiaojiao Pan\footnote{School of Mathematics, Nanjing University, 22 Hankou Road, Gulou District, Nanjing 210093, China, panjiaojiao@smail.nju.edu.cn.}}
	
	\date{}
	
	\maketitle
	
	\renewcommand{\refname}{References}
	
	\begin{abstract}
		This paper provides the first study of the homogenization of the 3D non-homogeneous  incompressible Navier--Stokes system in perforated domains with holes of supercritical size. The diameter of the holes is of order $\e^{\a} \ (1<\a<3)$, where $\e > 0$ is a small parameter measuring the mutual distance between the holes. We show that as $\e\to 0$, the asymptotic limit behavior of velocity and density is governed by Darcy's law under the assumption of a strong solution of the limiting system. Moreover, convergence rates are obtained. Finally, we show the existence of strong solutions to the inhomogeneous incompressible Darcy law, which might be of independent interest.
	\end{abstract}
	

	\tableofcontents
	
	
	\section{Introduction}\label{INTRODUCTION}
	Homogenization of different types of fluid flows in physical domains has been widely considered during the last decades. In the domains perforated by a large number of tiny holes, when the number of holes tends to infinity and their size tends to zero,  the limit system that describes the limit behavior of fluid flows is called homogenized system.
	
	Firstly, Tartar \cite{Tartar1} considered the homogenization of Stokes equations, where the holes' mutual distance is of the same order as their radius, and derived Darcy's law. For steady Stokes and Navier--Stokes
	equations,  Allaire \cite{ALL-NS1, ALL-NS2} provided a systematic study, which shows that the homogenized system is determined by the ratio $\s_{\e}$ given as
	\ba
	\sigma _\e: = \left(\frac{\e^d}{a_\e^{d-2}}\right)^{\frac{1}{2}},  \ d \geqslant 3;\quad{\sigma _\e}: = \e \left| \log \frac{{a_\e }}{\e} \right|^{\frac{1}{2}}, \ d = 2,\nn
	\ea
	where $\e > 0$ is the holes' mutual distance, and $a_\e > 0$ their radius. Roughly speaking, in the supercritical case, i.e., $\lim_{\e\to 0}\s_{\e}=0$, the holes are too large and the fluid velocity finally tends to zero, however, a proper rescaling shows that the limit behavior is governed by Darcy's law;  in the subcritical case, i.e., $\lim_{\e\to 0}\s_{\e}=\infty$, the holes are too tiny to be felt by the fluid and the limit system coincides with the original one; finally, in the critical case, i.e., $\lim_{\e\to 0}\s_{\e}=\s_{*}\in (0,+\infty)$, the holes impose some friction in the asymptotic limit, giving rise to an additional term ``out of nowhere'' (cf. \cite{CioranescuMurat1982}), and the limiting system is given by Brinkman's law\footnote{This additional term coincides with the well-known Stokes drag term, see \cite{HoeferKowalczykSchwarzacher2021} for a heuristic explanation.}. A unified approach to all the regimes mentioned was given by Lu in \cite{L}.
	
	The first homogenization of the compressible Navier--Stokes system was given by Masmoudi \cite{Mas-Hom}, who studied the case when the size of holes is proportional to their mutual distance, and derived Darcy's law. For the subcritical case in three spatial dimensions, there is a series of studies \cite{BO2,DFL,FL1,Lu-Schwarz18, OschmannPokorny2023}, which showed that the limit system is the same as the original one. Similar results where given in \cite{NP} and \cite{NO} for the two dimensional case. At low Mach number, \cite{BO1, HoeferKowalczykSchwarzacher2021} studied the critical and supercritical cases, respectively, and \cite{BCh24} investigated the full Navier--Stokes--Fourier system for heat-conducting fluids in the subcritical case.
	
	For the homogenization of the evolutionary incompressible Navier--Stokes system, Mikeli\'{c} \cite{Mik91} studied the case where the size of holes is proportional to their mutual distance and obtained Darcy's law. The critical and both sub- and supercritical regimes were considered in \cite{FeNaNe, LY}, respectively.

	Moreover, more complicated models describing fluid flows in homogenization have been studied: Feireisl, Novotn\'y, and Takahashi \cite{FNT-Hom} studied  the supercritical case of the full Navier--Stokes--Fourier system; Lu and Pokorn\'y \cite{Lu-Pokorny} considered the subcritical case of stationary compressible Navier--Stokes--Fourier equations; Feireisl, Lu, and Sun \cite{FLS} investigated the critical case of a non-homogeneous incompressible and heat conducting fluid confined to a 3D domain perforated. Jing \cite{Jing2020,Jing2021} considered Laplace equations and Lam\'e systems separately. Recently, Lu and Qian \cite{LY-Qian} considered the homogenization of evolutionary incompressible viscous non-Newtonian flows of Carreau-Yasuda type, where the size of holes is proportional to their mutual distance, and derived the linear Darcy's law.
	
	Regarding the convergence rates from the original system to the limit system, Allaire \cite{ALL-NS1,ALL-NS2} gave proper error estimates for Stokes equations for different sizes of the holes. Recently, Shen \cite{Shen22} gave the sharp convergence rates for Stokes equations to Darcy's Law. H\"{o}fer \cite{Hoefer2022} investigated the homogenization of the Navier--Stokes equations in perforated domains in the inviscid limit with different sizes of holes and obtained corresponding quantitative estimates. Finally, Lu and the second author \cite{Lu-Oschmann} generalized the results of \cite{LY-Qian} and derived Darcy's law when the size of the holes is $\e^{\a}$, $\a\in (1,\frac{3}{2})$. Also there, quantitative convergence rates are given.
	
	\subsection{Problems for supercritical case and aim of the paper}
	Although, as seen above, there is a vast of papers concerning homogenization of incompressible and compressible flows, there does not seem to be much literature regarding the ``in-between-case'' of inhomogeneous incompressible fluids. The case of critically sized holes was considered in \cite{Pan2024}, whereas the subcritical case was recently solved in \cite{LuPanYang2025}. Therefore, in the present paper, we wish to complete the study in considering the homogenization of the 3D non-homogeneous incompressible Navier--Stokes system, where the diameter of the holes is of order $\e^{\a} \ (1<\a<3)$, and the mutual distance of holes is proportional to $\e$ (supercritical case).
	
	Due to the presence of density, the traditional method of compactness arguments and oscillating test functions is not enough for us to obtain the limit equation. The main problem lies, both surprisingly and evidently, in the continuity equation: surprisingly, because for compressible fluids, usually the momentum equation causes problems, specifically the non-linear pressure and convective term (cf. \cite{FL1, Lu-Schwarz18, OschmannPokorny2023}); evidently, because the pressure does not depend (explicitly) on the density, hence no compactness argument is required there, and the fluid velocity is small such that the (only remaining non-linear) convective term finally vanishes. Specifically, our main problem is that the mere weak convergence of density $\vr_\e$ and velocity $\uu_\e$ do not give enough information to pass to the limit of the product $\vr_\e \uu_\e$ in the continuity equation. Instead, we will use the method of relative entropy (energy), which has been shown to be a successful tool when showing weak-strong uniqueness results as well as convergence rates, see \cite{Tartar1, Daus20, CBSV24, HNO24, Lu-Oschmann}. To this end, we will properly modify the velocity $\uu$ and density $\vr$ of the limiting system and compare this modification with the original primitive density-velocity-couple $(\vr_\e, \uu_\e)$. In doing so, we do not just prove the strong convergence of $(\vr_\e, \uu_\e)$ to $(\vr, \uu)$, but also give precise convergence rates.
	
Let us emphasize that the literature for the supercritical case in the presence of density is indeed rather sparse. Up to the authors' knowledge, the non-homogeneous incompressible case has not been addressed yet. Moreover, there seem to be just two rigorous papers for the compressible case: the seminal one by Masmoudi \cite{Mas-Hom} for the qualitative result if $\a = 1$, and the paper by H\"ofer, Ne\v{c}asov\'a, and the second author \cite{HNO24}, recovering quantitatively Masmoudi's result with convergence rates and generalizing it to the case of a small Knudsen number. On the other hand, the case $\a>1$ seems to be completely out of reach with the present methods; although the result from \cite{HoeferKowalczykSchwarzacher2021} gives convergence for all $1<\a<3$, the limiting system is indeed incompressible, which is due to the assumption of a low Mach number vanishing as $\e$ does so. Based on this lack of literature, in the present paper, we give the first step in the direction of a limiting system having non-constant density.
	
	\paragraph{{\bf Organization of the paper:}} The paper is organized as follows. In Section~\ref{sec:probForm}, we introduce the problem under consideration, give the main results, and introduce the local problem as a crucial ingredient in the homogenization process later on. Section~\ref{sec:unifbds} is devoted to show uniform bounds on the velocity and density, and to extend the pressure in a suitable way. The homogenization in case of a toroidal domain is carried out in Section~\ref{sec:conv}, and adaptations for bounded domains are made in Section~\ref{sec:bdDom}. The convergence rates for the pressure can be found in Section~\ref{sec:pressConv}. Finally, the local existence of strong solutions for the target system is provided in Section \ref{sec:ExistenceStrong}.
	
	\paragraph{{\bf Notations:}} We use standard notations $L^p$ and $W^{k,p}$ for Lebesgue and Sobolev spaces, respectively. If no ambiguity occurs, we denote them like in the scalar case even for vector- or matrix-valued functions, that is, $L^p(\O)$ instead of $L^p(\O; \R^3)$. The Frobenius inner product of two matrices $A,B \in \R^{3 \times 3}$ is denoted as $A:B = \sum_{i, i=1}^3 A_{ij} B_{ij}$. Moreover, we write $a \lesssim b$ whenever there is a constant $C>0$ that is independent of $a, b,$ and $\e$ such that $a \leq C b$. For a function $f$ defined in some domain $D \subset \R^3$, we use the symbol $\tilde{f}$ to denote the zero-extension of $f$ to $\R^3$, that is,
	\be
	\tilde{f}=f \ \ \mbox{in} \ \ D, \ \ \ \tilde{f}=0 \ \ \mbox{in} \ \ \R^3 \setminus D. \nn
	\ee
	Moreover, we use the symbol $L^{q}_0(D)$ to denote the space of functions in $L^q(D)$ with zero integral mean:
	\be
	L^q_0(D):=\left\{f\in L^q(D): \, \int_{D} f\,\dx=0\right\}. \nn
	\ee
	
	\section{Problem formulation, main results, and useful tools}\label{sec:probForm}
	\subsection{Problem formulation}\label{Problem formulation}
	We consider the three-dimensional torus $\O = \mathbb{T}^3 = (\R / \Z)^3$, or a bounded domain $\O \subset \R^3$ of class $C^{2,\b}$ with $0<\b<1$, and introduce a family of $\e$-dependent perforated domains
	$\{ \O_\e \}_{\e \in (0, 1)}$ as
	\begin{align}\label{OmegaEps}
	\O_\e = \O \setminus \bigcup_{k\in K_\e} T_{\e, k},\ && K_\e:=\{k\in \mathbb{Z}^3|\ \overline{Q}_{k, \e} \subset \O \}, && Q_{k, \e} := \e \Big( -\frac{1}{2},\frac{1}{2} \Big)^3 + k,\ \ k\in \mathbb{Z}^3,
	\end{align}
	where the sets $T_{\e,k}$ represent holes or obstacles. Suppose we have the following property concerning the distribution of the holes:
	\be
	T_{\e,k} := x_{\e,k} + a_\e T_0 \Subset B(x_{\e,k}, b_0a_{\e} ) \Subset  Q_{k, \e} \subset \O, \nn
	\ee
	with
	\be
	|x_{\e, k} - x_{\e, l}| \geq 2\e \ \mbox{ for any } \ k \neq l \in K_\e. \nn
	\ee
	Here, $b_0$ is a positive constant independent of $\e$, $0 \in T_0 \subset \R^3$ is a bounded simply connected domain of class $C^{2,\b}$ contained in $Q_{0,1}$, and $B(x,r)$ denotes the open ball centered at $x$ with radius $r>0$. Without loss of generality, we will assume that $T_0 \subset B(0, \frac18)$. The diameter of each $T_{\e,k}$ is of size $\mathcal{O}(a_\e)$ and their mutual distance  is $\mathcal{O}(\e)$. For the torus case, we will restrict ourselves to $\e \in (0, 1)$ with $(2\e)^{-1} \in \Z$ such that $\mathbb{T}^3 = \O = \cup_{k \in K_\e} Q_{k, \e}$. In either case, the number of holes contained in $\O$ satisfies
	\be
	|K_\e| \leq C(\e^{-3} |\Omega| + o(1)),\quad \mbox{for some} ~C>0~ \mbox{independent of} ~ \e.\nn
	\ee
	
	For given $T>0$, we now consider the following non-homogeneous incompressible Navier--Stokes system in $(0,T) \times \O_{\e}$:
	\ba\label{NSE}
	\begin{cases}
		\d_{t}\vr_{\e}+\dive (\vr_{\e} \vu_{\e}) = 0, &\mbox{in } (0,T) \times \O_{\e},\\
		\dive \vu_{\e} = 0, &\mbox{in } (0,T) \times \O_{\e},\\
		\s_\e^4 \big( \d_{t}(\vr_{\e} \vu_{\e})+ \dive (\vr_{\e} \vu_{\e} \otimes \vu_{\e}) \big) - \s_\e^2 \mu \Delta_x \vu_{\e} +\nabla_x p_{\e} =\vr_{\e} \vf, & \mbox{in } (0,T) \times \O_{\e},\\
		\vu_{\e} = 0, \quad &\mbox{on } (0,T) \times \d\O_{\e},\\
		\vr_{\e}|_{t=0} = \vr_{\e}^{0},\quad (\vr_{\e}\vu_{\e})|_{t=0} = \vr_{\e}^{0}\vu_{\e}^{0} & \mbox{in } \O_\e,
	\end{cases}
	\ea
	where $\vr_{\e}$ is the fluid's mass density, $\vu_{\e}$ its velocity field in $\R^{3}$, $p_{\e}$ the pressure, $\mu > 0$ is a constant viscosity coefficient, and $\vf \in L^{\infty}(0,T; L^\infty(\O))$ is the external force density. The boundary condition on the outer boundary $\d \O$ is redundant if $\O = \mathbb{T}^3$ is the three-dimensional torus.\\
	
	We remark that the presence of the factor $\s_\e^4$ comes from the fact that we anticipate the \emph{original} fluid velocity to be small of order $\s_\e^2$ and scale this factor out, such that the \emph{rescaled} fluid velocity $\uu_\e = \mathcal{O}(1)$ (in some sense); see \cite{HNO24, Lu-Oschmann} for more details on this rescaling argument.
	\subsection{Weak solutions}\label{Weak solutions}
	Let us now introduce the definition of finite energy weak solutions to the non-homogeneous incompressible Navier--Stokes system \eqref{NSE}.
	\begin{defi}\label{fews}
		Let
		\ba\label{init}
		\vr_\e^0 \in L^\infty(\O_\e), \quad \vr_\e^0 \uu_\e^0 \in L^2(\O_\e).
		\ea
		
		A couple $(\vr_{\e},\vu_{\e})$ is called a \emph{finite energy weak solution} to system \eqref{NSE} in $(0,T) \times \O_{\e}$ provided:
		
		\begin{enumerate}
			
			\item  $(\vr_{\e},\vu_{\e})$ satisfies:
			\begin{gather*}
				\vr_{\e}\in L^{\infty}(0,T;L^{\infty}({\O_{\e}}))\cap C([0,T],L^{2}({\O_{\e}})), \quad \vr_{\e}\geq 0 \quad \mbox{a.e. in} \quad (0,T) \times \O_{\e},\\
				\sqrt{\vr_{\e}}\vu_{\e}\in L^{\infty}(0,T;L^{2}({\O_{\e}})),\quad \vr_{\e}\vu_{\e} \in C_{\textup{weak}}([0,T],L^{2}({\O_{\e}})), \\
				\vu_{\e}\in L^{2}(0,T;W_{0}^{1,2}({\O_{\e}})).
			\end{gather*}
			
			Moreover,
			\begin{align}\label{consLp}
				\begin{cases}
					\int_{\O_{\e}} |\vr_{\e}(t, x)|^{q} \dx = \int_{\O_{\e}} |\vr_{\e}^{0}(x)|^{q} \dx, \quad \mbox{for any } q\in[1,\infty) \mbox{ and } t\in [0,\infty), \\
					\|\vr_{\e}\|_{L^{\infty}(0,T;L^{\infty}({\O_{\e}}))}=\|\vr_{\e}^{0}\|_{L^{\infty}({\O_{\e}})}.
				\end{cases}
			\end{align}

			\item For any $\psi\in C_{c}^{\infty}([0, T) \times \overline{\O_{\e}})$, 
			\be\label{continuity eq}
			\int_{0}^{T}\int_{\O_{\e}}\vr_{\e}(\d_{t}\psi+\vu_{\e}\cdot\nabla_x\psi) \dx \dt = -\int_{\O_{\e}}\vr_{\e}^{0}\psi(0,x) \dx.
			\ee
			
			\item For any  $\bvp \in C_{c}^{\infty}([0,T) \times \O_{\e}; \R^{3})$,
			\ba\label{momentum eq}
			&\int_{0}^{T}\int_{\O_{\e}} \big[-\s_{\e}^{4} \vr_{\e}\vu_{\e}\cdot\d_{t} \bvp- \s_\e^4 (\vr_{\e} \vu_{\e} \otimes \vu_{\e}):\nabla_x \bvp + \s_\e^2 \mu \nabla_x \vu_{\e}:\nabla_x \bvp + p_{\e} \dive \bvp \big] \dx \dt \\
			&= \int_{0}^{T}\int_{\O_{\e}}\vr_{\e}\vf \cdot \bvp \dx \dt +\int_{\O_{\e}}\s_{\e}^{4} \vr_{\e}^{0} \vu_{\e}^{0}\cdot \bvp(0,x) \dx.
			\ea
			
			\item For a.a. $\tau \in(0,\infty)$, there holds the energy inequality
			\ba\label{enIneq}
			&\frac{1}{2}\s_{\e}^{4}\int_{\O_{\e}}\vr_{\e} |\vu_{\e}(\tau, x)|^{2} \dx + \s_\e^2 \mu\int_{0}^{\tau}\int_{\O_{\e}}|\nabla_x \vu_{\e}|^{2} \dx \dt \\
			&\leq \frac{1}{2}\s_{\e}^{4}\int_{\O_{\e}}\vr_{\e}^{0} |\vu_{\e}^{0}(x)|^{2} \dx +\int_{0}^{\tau}\int_{\O_{\e}}\vr_{\e}\vf \cdot \vu_{\e} \dx \dt.
			\ea
		\end{enumerate}
	\end{defi}
	
	For any fixed $\e>0$, the global existence of finite energy weak solutions in Definition~\ref{fews} is shown in \cite[Theorem~2.1]{Lions1997}.

	\subsection{Main results}
	In this paper, we focus on the three dimensional case and take
	\be
	a_{\e}=\e^{\a} \quad\mbox{with}\quad 1<\a<3,\nn
	\ee
	which implies $\s_{\e}=\e^{\frac{3-\a}{2}}$. \\
	
	Since the fluid is incompressible and the holes are large, we expect in the limit the same behavior as for the purely incompressible case considered by Allaire \cite{ALL-NS2} (see also \cite{HNO24} for the compressible case with $\a = 1$). More precisely, we expect the limiting equations to be
	\begin{align}\label{Darcy}
		\begin{cases}
			\d_t \vr + \dive(\vr \uu) = 0 & \mbox{in } (0,T) \times \O,\\
			\dive \vu=0 & \mbox{in } (0,T) \times \O,\\
			\mu M_{0}\vu=\vr \vf-\nabla_x p & \mbox{in } (0,T) \times \O,\\
			\vu\cdot \mathbf n=0 & \mbox{on } (0,T) \times \d\O,
		\end{cases}
	\end{align}
	where $M_0$ is the resistance matrix defined via the so-called local problem, see \eqref{M0} below. Again, the boundary condition on $\d \O$ is redundant if $\O = \mathbb{T}^3$.\\
	
	Throughout the paper, we will assume the zero extensions of the initial data to satisfy
	\ba\label{init2}
	\tilde \vr^{0}_{\e} \in L^\infty(\O), \ \s_\e^2 \sqrt{\tilde \vr_\e^0} \tilde \vu^{0}_{\e} \in L^{2}(\O) \mbox{ are uniformly bounded in } \e.
	\ea
	
	We are now in the position to give the main results of this paper.
	
	\begin{theorem}[Torus case]\label{limitTorus}
		Let $\Omega = \mathbb{T}^3$ be the three-dimensional torus and $1<\a<3$. Let $(\vr_{\e},\vu_{\e})$ be a finite energy weak solution of the Navier--Stokes system \eqref{NSE} in the sense of Definition~\ref{fews} with initial data satisfying \eqref{init2}, and force $\vc f \in L^\infty(0,T; L^\infty(\O))$. Moreover, let $(\vr, \vu)$ be a strong solution to Darcy's law \eqref{Darcy} belonging to the class
		\begin{align*}
			 \vr \in C([0,T]; W^{3,2}(\Omega)), &\quad \partial_t \vr \in L^2(0,T; W^{2,2}(\Omega)), \\
			\vu \in L^2(0,T; W^{3,2}(\Omega)), &\quad \partial_t \vu \in L^2(0,T; W^{2,2}(\Omega)). 
		\end{align*}
		Then, there is $\e_0>0$ and a constant $C>0$ which is independent of $\e$ such that for any $\e \in (0, \e_0)$,
		\begin{align*}
			&\|\tilde \uu_\e - \uu\|_{L^2((0,T) \times \O)}^2 + \|\tilde \vr_\e - \vr\|_{L^2((0,T) \times \O)}^2 \\
			&\leq C \big( \sigma_\e^4 \|\sqrt{\tilde \vr_\e^0} ( \tilde{\uu}_{\e}^0 - \uu(0, \cdot) )\|_{L^2(\O)}^2 + \|\tilde \vr_\e^0 - \vr(0, \cdot)\|_{L^2(\O)}^2 + \e^{\a - 1} + \e^{3 - \a} \big).
		\end{align*}
	\end{theorem}
	
	For bounded domains, we have a similar result:
	\begin{theorem}[Case of bounded domain]\label{limitDomain}
		Let $\Omega \subset \R^3$ be a bounded domain of class $C^{2, \beta}$, $\beta \in (0,1)$, and $1<\a< \frac{3}{2}$. Let otherwise the assumptions of Theorem~\ref{limitTorus} be in order. Then,
		\begin{align*}
			&\|\tilde \uu_\e - \uu\|_{L^2((0,T) \times \O)}^2 + \|\tilde \vr_\e - \vr\|_{L^2((0,T) \times \O)}^2 \\
			&\leq C \big( \sigma_\e^4 \|\sqrt{\tilde \vr_\e^0} ( \tilde{\uu}_{\e}^0 - \uu(0, \cdot) )\|_{L^2(\O)}^2 + \|\tilde \vr_\e^0 - \vr(0, \cdot)\|_{L^2(\O)}^2 + \e^{\a - 1} + \e^{3-2\a} \big).
		\end{align*}
	\end{theorem}
	
	\begin{remark}
		The different convergence rates stem from the mismatch of the boundary conditions between $\uu_\e$ and $\uu$ on $\d \O$, namely $\uu_\e |_{\d \O} = 0$ versus $\uu \cdot \vc n |_{\d \O} = 0$. To overcome this issue, we will construct a boundary layer corrector, giving precisely the worse convergence rates in the case of a bounded domain. For more details on this, we refer to \cite{Shen22, HNO24}.
	\end{remark}

	\begin{remark}
	We remark that in order to show our results, the regularity classes required in Theorem~\ref{limitTorus} can be weakened to
	\begin{gather*}
	\vr \in L^2(0,T; L^3(\O)), \quad \d_t(\vr^2) \in L^1(0,T; L^\frac{3}{2}(\O)), \quad	\Grad \vr \in L^\infty(0,T; L^\infty(\O)), \\ 
	\uu \in L^2(0,T; W^{1,\infty}(\O)), \quad \d_t \uu \in L^1(0,T; W^{2,2}(\Omega)).
	\end{gather*}
	\end{remark}
	
	\subsection{Local problem and modified Poincar\'e inequality}\label{Local problem}
	A crucial ingredient in the homogenization process will be the so-called local problem, which we define as
	\begin{align*}
		\begin{cases}
			-\Delta_x \vv^{i}+\nabla_x q^{i}=0 & \mbox{in } \R^{3} \setminus T_0,\\
			\dive  \vv^{i}=0 & \mbox{in } \R^{3} \setminus T_0,\\
			\vv^{i}=0 & \mbox{on } \d T_0,\\
			\vv^{i}=\vc e^{i} & \mbox{at infinity},
		\end{cases}
	\end{align*}
	where $\{\vc e^{i}\}_{i=1,2,3}$ is the canonical basis of $\R^{3}$. 
	In order to go from $\R^3 \setminus T_0$ to the domain $\O_\e$, we define functions $(\vv^{i}_{\e},q^{i}_{\e})$ as follows: in cubes $Q_{k, \e}$ that intersect with the boundary of $\O$ (this is of course just relevant if $\O \neq \mathbb{T}^3$), 
	\be
	\vv^{i}_{\e}=\mathbf e^{i},\quad q^{i}_{\e}=0,\quad  \mbox{in} \quad Q_{k, \e}\cap\O \ \ \mbox{if} \ \ Q_{k, \e}\cap\d\O\neq\emptyset;\nn
	\ee
	in cubes $Q_{k, \e}$ whose closures are contained in $\O$,
	\begin{align*}
		\begin{cases}
			\vv^{i}_{\e}=\mathbf e^{i},\quad q^{i}_{\e}=0 & \mbox{in } Q_{k, \e}\backslash B(x_{k, \e},\frac{\e}{2}),\\
			-\Delta_x \vv^{i}_{\e}+\nabla_x q^{i}_{\e}=0,\quad \dive \vv_{\e}^{i}=0 & \mbox{in } B(x_{k, \e},\frac{\e}{2})\backslash B(x_{k, \e},\frac{\e}{4}),\\
			\vv^{i}_{\e}(x)=\vv^{i}\left(\frac{x-x_{k, \e}}{\e^{\a}}\right),\quad q^{i}_{\e}(x)=\frac{1}{\e^{\a}}q^{i}\left(\frac{x-x_{k, \e}}{\e^{\a}}\right) & \mbox{in } B(x_{k, \e},\frac{\e}{4})\backslash T_{\e,k},\\
			\vv^{i}_{\e}=0,\quad q^{i}_{\e}=0 & \mbox{in } T_{\e,k},
		\end{cases}
	\end{align*}
	together with matching (Dirichlet) boundary conditions on each interface. From \cite[Proposition~A.1]{Lu-Oschmann}, we can state the following
		\begin{proposition} \label{TestFct}
		Let $1< \alpha < 3$. For any $\e \in (0,1)$ fixed, set the tensor-valued function $W_{\e}=(\vv_{\e}^{1}, \vv_{\e}^{2}, \vv_{\e}^{3})$ and the vector-valued function $\vq_{\e}=(q_{\e}^1, q_{\e}^2, q_{\e}^3)^T$. Then 
		\begin{align}
			W_{\e}\in W^{1,\infty}(\Omega; \mathbb{R}^{3\times 3}), \quad  &\vq_{\e} \in L^{\infty}(\Omega; \mathbb{R}^3), \notag \\
			\dive \vv_{\e}^i=0 \mbox{ in } \Omega,\quad  &\vv_{\e}^i=0 \mbox{ in } T_{\e, k}, \label{c2}
		\end{align}
		for any $i \in \{1,2,3\}$ and any $k \in K_\e$, and
		\begin{align}
			\| W_{\e} - \Id\|_{L^p(\Omega)} \lesssim \e^{\min \left\{1, \frac{3}{p}\right\}(\alpha-1)}\quad &\mbox{for any} \quad  p \in [1,\infty], \label{estVei}\\[0.1cm]
			\| \Grad \vv_{\e}^i\|_{L^p(\Omega)} + \| q_{\e}^i \|_{L^p(\Omega)} \lesssim \e^{\frac3p(\alpha-1)-\alpha} \quad &\mbox{for any} \quad p \in (\frac32, \infty], \label{estGradv}
		\end{align}
		for any $i\in \{1,2,3\}$.
		
		Moreover, for any sequence $\{ \vw_{\e} \}_{\e>0} \subset W^{1,2}(\Omega)$ satisfying $\vw_{\e} =0$ in $\bigcup_{k \in K_\e} T_{\e, k}$, 
		\begin{equation*}
			\sigma_{\e}\| \Grad \vw_{\e} \|_{L^2(\Omega)} \lesssim 1,
		\end{equation*}
		\begin{equation*}
			\vw_{\e} \rightharpoonup \vw \quad \mbox{weakly in } L^2(\Omega),
		\end{equation*}
		and for any $\varphi \in \mathcal{D}(\Omega)$, we have 
		\begin{equation} \label{M1}
			\sigma_{\e}^2 \langle \Grad q_{\e}^i - \Delta_x \vv_{\e}^i; \ \varphi \vw_{\e} \rangle_{W^{-1, 2}(\Omega), W^{1,2}(\Omega)} \to \int_{\Omega} \varphi \ (M_0 \textbf{\textup{e}}^i) \cdot \vw \ \dx, 
		\end{equation}
		where the (symmetric and positive definite) resistance matrix $M_0=[M_{ij}]_{i,j=1,2,3}$ is defined as 
		\begin{equation*}
			\langle M_{ij}, \varphi \rangle_{\mathcal{D}', \mathcal{D}} = \lim_{\e \to 0} \sigma_{\e}^2 \int_{\Omega} \varphi \ \nabla_x \vv_{\e}^i: \nabla_x \vv_{\e}^j \ \dx, \quad \mbox{for any } \varphi \in \mathcal{D}(\Omega).
		\end{equation*}
		Lastly, there holds the error estimate
		\begin{equation}\label{errorM}
		\|\sigma_\e^2(\Grad \vq_\e - \Delta_x W_\e) - M_0\|_{W^{-1,2}(\Omega)} \lesssim \e.
		\end{equation}
	\end{proposition}

\begin{remark}
In terms of the local functions $(\vv^i, q^i)$, the matrix $M_0$ can be directly computed via
\begin{align}\label{M0}
M_{ij} = \int_{\R^3 \setminus T_0} \Grad \vv^i : \Grad \vv^j \dx,
\end{align}
showing in particular that $M_0$ is symmetric and positive definite, hence invertible\footnote{The inverse $K_0 = M_0^{-1}$ is usually known as the \emph{permeability} of the medium.}.
\end{remark}	
	
	
	In the case of large holes, one can benefit from the zero boundary condition on the holes and obtain the following improvement of Poincar\'e's inequality (see \cite[Lemma~3.4.1]{ALL-NS2}):
	\begin{lemma}
		Let $\O_{\e}$ be the perforated domain defined by \eqref{OmegaEps} with $\e \in (0,1)$, $a_{\e}=\e^{\a}$, $1<\a<3$, and $\s_{\e}=\e^{\frac{3-\a}{2}}$. There exists a constant $C$ independent of $\e$ such that
		\be\label{poinc}
		\|\vu\|_{L^{2}(\O_\e)}\leq C \s_{\e} \|\nabla_x \vu\|_{L^{2}(\O_{\e})},\quad \mbox{for any}\quad \vu\in W_{0}^{1,2}(\O_\e).
		\ee
	\end{lemma}
	
	\section{Uniform bounds}\label{sec:unifbds}
	\subsection{Bounds for the velocity}
	
	Using the energy inequality \eqref{enIneq} and the assumptions on the initial data \eqref{init2}, we obtain
	\ba\label{jedna}
	&\s_\e^4 \int_{\O_\e} \vr_\e |\uu_\e|^2(\tau) \dx + \mu \s_\e^2 \int_0^\tau \int_{\O_\e} |\nabla_x \uu_\e|^2 \dx \dt \\
	&\leq \s_\e^4 \int_{\O_\e} \vr_\e^0 |\uu_\e^0|^2 \dx + \int_0^\tau \int_{\O_\e} \vr_\e \vf \cdot \uu_\e \dx \dt \\
	&\leq \s_\e^4 \|\sqrt{\vr_\e^0} \uu_\e^0\|_{L^2(\O_\e)}^2 + \|\vr_\e\|_{L^2((0,T) \times \O_\e)} \|\vf\|_{L^\infty(0,T; L^\infty(\O_\e))} \|\uu_\e\|_{L^2((0,T) \times \O_\e)} \\
	&\leq C + C \s_\e \|\nabla_x \uu_\e\|_{L^2((0,T) \times \O_\e)} \\
	&\leq C + \frac{\mu}{2} \s_\e^2 \|\nabla_x \uu_\e\|_{L^2((0,T) \times \O_\e)}^2,
	\ea
	where we used that $\|\vr_\e\|_{L^2((0,T) \times \O_\e)} = \|\vr_\e^0 \|_{L^2(\O_\e)} \leq C$ independently of $\e$. Absorbing the very last term by the left-hand side of \eqref{jedna} yields by Poincar\'e's inequality \eqref{poinc}
	\ba\label{unifBds}
	\s_\e^2 \|\sqrt{\vr_\e} \uu_\e\|_{L^\infty(0,T; L^2(\O_\e))} + \|\uu_\e\|_{L^2((0,T) \times \O_\e)} + \s_\e \|\nabla_x \uu_\e\|_{L^2((0,T) \times \O_\e)} \leq C.
	\ea
	
	In particular, $\|\uu_\e\|_{L^2((0,T) \times \O_\e)} = \mathcal{O}(1)$, and thus up to subsequences,
	\begin{align*}
		\tilde \uu_\e \weak \uu \mbox{ weakly in } L^2((0,T) \times \O).
	\end{align*}
	
	In order to proceed, we will follow an idea of Mikeli\'c (see \cite{Mik91}) to consider the time-integrated variables and equations instead of the original ones; therefore, we set
	\begin{align}\label{UGF}
		\vU_\e = \int_0^t \uu_\e(s, \cdot) \, {\rm d}s, && \Psi_\e=\int_0^t  (\vr_\e \uu_\e \otimes \uu_\e)(s, \cdot)\, {\rm d}s, && \mathbf{F}_\e=\int_0^t (\vr_\e \mathbf{f})(s, \cdot) \, {\rm d}s.
	\end{align}
	It follows that $\dive \vU_\e = 0$ and 
	\begin{align*}
		\vU_\e \in C([0,T];\ W^{1,2}_0 (\O_\e)), && \Psi_\e \in C([0,T]; L^{3}(\O_\e)), && \mathbf{F_{\e}}\in C([0,T]; L^2(\O_\e)).
	\end{align*}
	Moreover, from the weak convergence $\tilde{\uu}_\e \weak \uu$ we deduce that 
	\ba\label{cov-Ue}
	\tilde{\vU}_\e \weak \vU = \int_{0}^{t} \uu(s, \cdot) \,{\rm d}s  \text{ weakly in } L^2((0,T) \times \O).
	\ea
	
	Combining this with the definition of $\vU_\e$ immediately gives
	\begin{align*}
		\|\vU_\e\|_{W^{1,2}(0,T; L^2(\O_\e))} + \s_\e \|\nabla_x \vU_\e\|_{W^{1,2}(0,T; L^2(\O_\e))} \leq C.
	\end{align*}
	
	Additionally, classical theory of Stokes equations (see \cite[Chapter~3]{Temam77}) ensures the existence of
	\begin{equation*}
		P_\e \in C_{\rm weak}([0,T]; L^2_{0}(\O_\e))
	\end{equation*}
	such that for each $t\in [0,T]$,
	\begin{equation}\label{Pe-def}
		\nabla_x P_\e=\mathbf{F}_\e- \s_\e^4 (\uu_\e-\uu_{\e0}) - \s_\e^4 \dive \Psi_\e + \s_\e^2 \mu \Delta_x \vU_\e \quad \mbox{in}  \ \mathcal{D}'(\O_{\e}).
	\end{equation}
	
	\subsection{Pressure extension}\label{sec:pressExt}
	Uniform pressure bounds are not as easy to obtain. A naive idea would be to set $\tilde{P}_\e = \int_0^t \tilde{p}_\e(s, \cdot) \ {\rm d}s$, however, we do not know whether $\tilde{p}_\e$ is uniformly bounded in $L^2((0,T) \times \O)$. In order to circumvent this issue, we use a more subtle way to extend the pressure from $\O_\e$ to $\O$, and follow \cite[Lemma~3.2]{Mik91}. To this end, we state the existence of a suitable restriction operator:
	\begin{lemma}\label{Restriction operator}
		For $\O_{\e}$ defined in Section \ref{Problem formulation}, there exists a bounded linear restriction operator $R_{\e}: W_{0}^{1,2}(\O;\R^{3})\rightarrow W_{0}^{1,2}(\O_\e;\R^{3})$ such that
		\be
		\vu\in W_{0}^{1,2}(\O_{\e}; \R^{3})\Longrightarrow R_{\e}(\tilde{\vu})=\vu\ \ \mbox{in} \ \ \ \O_{\e},\nn
		\ee
		\be
		\vu\in W_{0}^{1,2}(\O; \R^{3}),\quad \dive\vu=0\ \ \mbox{in}\ \ \O\Longrightarrow \dive R_{\e}(\vu)=0\ \ \mbox{in} \ \ \ \O_{\e},\nn
		\ee
		\be
		\vu\in W_{0}^{1,2}(\O; \R^{3})\Longrightarrow\|\nabla_x R_{\e}(\vu)\|_{L^{2}(\O_\e)}\leq C(\|\nabla_x \vu\|_{L^{2}(\O)} + \s_{\e}^{-1} \|\vu\|_{L^{2}(\O)}).\nn
		\ee
		For each $\bvp \in L^{q}(0,T;W_{0}^{1,2}({\O}))$ with $1<q<\infty$, the restriction $R_{\e}(\bvp)$ is taken only on the spatial variable, which means 
		\be
		R_{\e}(\bvp)(\cdot,t) = R_{\e}(\bvp(\cdot,t))(\cdot), \quad \mbox{for any} \quad t\in (0,T).\nn
		\ee
		Then, $R_{\e}$ maps ${L^{q}(0,T;W_{0}^{1,2}(\O; \R^3))}$ onto ${L^{q}(0,T;W_{0}^{1,2}(\O_{\e}; \R^3))}$ and there holds the estimate 
		\be\label{restOp1}
		\|R_\e(\bvp)\|_{L^q(0,T; L^2(\O_\e))} + \s_\e \|\nabla_x R_{\e}(\bvp)\|_{L^{q}(0,T;L^{2}({\O_{\e}}))} \leq C ( \|\bvp\|_{L^q(0,T; L^2(\O))} + \s_\e \|\nabla_x \bvp\|_{L^{q}(0,T;L^{2}({\O}))} ).
		\ee
	\end{lemma}
	\begin{proof}
		See \cite[Proposition~3.4.10]{ALL-NS2}.
	\end{proof}
	
	Note in particular that inequality \eqref{restOp1} implies by $\s_\e < 1$ and Poincar\'e inequality in $\O$
\begin{align}\label{restOp}
\|R_\e(\bvp)\|_{L^q(0,T; L^2(\O_\e))} + \s_\e \|\nabla_x R_{\e}(\bvp)\|_{L^{q}(0,T;L^{2}({\O_{\e}}))} \leq C \|\nabla_x \bvp\|_{L^{q}(0,T;L^{2}({\O}))}
\end{align}	
	 which will be the crucial estimate in this section. Furthermore, by \cite[Lemma~2.1]{LY}, we have the following
	\begin{lemma}\label{definition of P}
		Let $1<q<\infty$ and assume $H\in L^{q}(0,T;W^{-1,2}({\O;\R^{3}}))$ satisfies 
		\be
		\langle  H, \bvp \rangle_{(0,T) \times \O}=0,\quad \forall \bvp \in C_{c}^{\infty}((0,T) \times \O;\R^{3}), \quad \dive \bvp =0.\nn
		\ee
		Then there exists a scalar function $P\in {L^{q}(0,T;L_{0}^{2}({\O}))}$ satisfying
		\be
		H=\nabla_x P, \quad \mbox{with}\quad \|P\|_{L^{q}(0,T;L_{0}^{2}({\O}))}\leq C\|H\|_{L^{q}(0,T;W^{-1,2}({\O}))}.\nn
		\ee
	\end{lemma}

	Having the two foregoing lemmas at hand, we want to extend the pressure $p_\e$, defined in $\O_\e$, to the whole of $\O$ and show uniform bounds for this extension. To this end, define for each $t \in [0, T]$ a functional $\tilde H_{\e}\in \mathcal D'(\O)$ by
	\ba
	&\langle \tilde H_{\e}(t) , \bvp \rangle_{\O}=\langle \nabla_x P_{\e}(t) ,R_{\e}(\bvp) \rangle_{\O_{\e}}\\
	&=\langle \mathbf F_{\e}(t)-\s_{\e}^{4}\vr_{\e} \vu_{\e}(t)+\s_{\e}^{4}\vr_{\e}^{0} \vu_{\e}^{0} - \s_\e^4 \dive \Psi_{\e}(t) + \s_\e^2 \mu \Delta_x \vU_{\e},R_{\e}(\bvp) \rangle_{\O_{\e}}, \nn
	\ea
	where $P_{\e}$ is given by \eqref{Pe-def}. By the uniform estimates \eqref{unifBds} obtained for $\uu_\e$, and the fact that $\|\vr_\e\|_{L^p(0,T; L^p(\O_\e))} = \|\vr_\e^0\|_{L^p(\O_\e)} \lesssim \|\vr_\e^0\|_{L^\infty(\O_\e)} \leq C$ for any $p \in [1, \infty]$, we find
	\ba\label{a}
	\s_\e^4 |\langle  \dive \Psi_{\e}(t),R_{\e}(\bvp) \rangle_{\O_{\e}}| &= \s_\e^4 |\langle \Psi_{\e}(t),\nabla_x R_{\e}(\bvp) \rangle_{\O_{\e}}|\\
	&\leq \s_\e^4 \|\Psi_{\e}(t)\|_{L^{\infty}(0,T;L^{2}({\O_{\e}}))}\cdot \|\nabla_x R_{\e}(\bvp)\|_{L^{2}({\O_{\e}})}\\
	&\leq C \s_\e^4 \|\vr_{\e} \vu_{\e} \otimes \vu_{\e}\|_{L^{1}(0,T;L^{2}({\O_{\e}}))}\cdot \s_\e^{-1} \|\nabla_x \bvp\|_{L^{2}({\O})}\\
	&\leq C \s_\e^3 \|\vr_{\e}\|_{L^{\infty}(0,T;L^{\infty}({\O_{\e}}))}\|\vu_{\e}\|_{L^{2}(0,T;L^{4}({\O_{\e}}))}^{2} \|\nabla_x \bvp\|_{L^{2}({\O})}\\
	&\leq C \s_\e^3 \|\vr_{\e}\|_{L^{\infty}(0,T;L^{\infty}({\O_{\e}}))}\|\vu_{\e}\|_{L^{2}(0,T;W^{1,2}({\O_{\e}}))}^{2} \|\nabla_x \bvp\|_{L^{2}({\O})}\\
	&\leq C\s_\e \|\nabla_x \bvp\|_{L^{2}({\O})}.
	\ea
Moreover,
	\ba
	\s_\e^2 |\langle \mu \Delta_x \vU_{\e},R_{\e}(\bvp) \rangle_{\O_{\e}} | &= \s_\e^2 |\langle \mu \nabla_x \vU_{\e},\nabla_x R_{\e}(\bvp) \rangle_{\O_{\e}}|\\
	&\leq C \s_\e^2 \|\nabla_x \vU_{\e}\|_{L^{\infty}(0,T;L^{2}({\O_{\e}}))} \|\nabla_x R_{\e}(\bvp)\|_{L^{2}(\O_{\e})}\\
	&\leq C \s_\e^2 \|\nabla_x \uu_{\e}\|_{L^{2}(0,T;L^{2}({\O_{\e}}))} \|\nabla_x R_{\e}(\bvp)\|_{L^{2}(\O_{\e})}\\
	&\leq C \|\nabla_x \bvp\|_{L^{2}(\O)},
	\ea
and  
	\ba
	|\langle \s_{\e}^4 \vr_{\e} \vu_{\e}(t),R_{\e}(\bvp) \rangle_{\O_{\e}}| &\leq C \s_{\e}^4 \|\sqrt{\vr_{\e}}\|_{L^{\infty}(0,T;L^{\infty}({\O_{\e}}))} \|\sqrt{\vr_\e}\vu_{\e}\|_{L^{\infty}(0,T;L^{2}({\O_{\e}}))}\| R_{\e}(\bvp)\|_{L^{2}({\O_{\e}})}\\
	& \leq C \s_{\e}^4 \|\sqrt{\vr_{\e}}\|_{L^{\infty}(0,T;L^{\infty}({\O_{\e}}))} \|\sqrt{\vr_\e} \vu_{\e}\|_{L^{\infty}(0,T;L^{2}({\O_{\e}}))}\cdot \s_{\e}\|\nabla_x R_{\e}(\bvp)\|_{L^{2}({\O_{\e}})}\\
	& \leq C \s_{\e}^2 \|\nabla_x \bvp\|_{L^{2}({\O})}.
	\ea
	Additionally,
	\ba
	|\langle \s_{\e}^{4} \vr_{\e}^{0} \vu_{\e}^{0},R_{\e}(\bvp) \rangle_{\O_{\e}}| &\leq C \s_{\e}^{4} \|\vr_{\e}^{0} \vu_\e^0\|_{L^{2}({\O_{\e}})} \cdot \| R_{\e}(\bvp)\|_{L^{2}({\O_{\e}})}\\
	& \leq C \s_{\e}^{4} \|\sqrt{\vr_\e^0}\|_{L^\infty(\O_\e)} \|\sqrt{\vr_{\e}^{0}} \vu_\e^0\|_{L^{2}({\O_{\e}})} \cdot \s_{\e}\|\nabla_x R_{\e}(\bvp)\|_{L^{2}({\O_{\e}})}\\
	& \leq C\s_{\e}^{2} \|\nabla_x \bvp\|_{L^{2}({\O})}.
	\ea
	Since 
	\ba
	\|\mathbf F_{\e}\|_{L^{\infty}(0,T;L^{2}({\O_{\e}}))}^{2} &= \sup_{t \in [0,T]} \int_{\O_{\e}} \left(\int_{0}^{t} (\vr_{\e} \vf)(\cdot,s) \, {\rm d}s \right)^{2} \dx \\
	&\leq C \|\vr_{\e}\|_{L^{\infty}(0,T;L^{\infty}({\O_{\e}}))}^{2}\| \vf_{\e}\|^{2}_{L^{2}(0,T;L^{2}({\O_{\e}}))}\leq C,\nn
	\ea
	we have
	\ba\label{b}
	&|\langle \mathbf F_{\e}(t),R_{\e}(\bvp) \rangle_{\O_{\e}}|\leq C\s_{\e} \|\mathbf F_{\e}\|_{L^{\infty}(0,T;L^{2}({\O_{\e}}))}\cdot \|\nabla_x R_{\e}(\bvp)\|_{L^{2}({\O_{\e}})}\\
	&\leq C \|\nabla_x \bvp\|_{L^{2}({\O})}.
	\ea
	
	Summing up \eqref{a}-\eqref{b}, we conclude that for any $t \in [0,T]$,
	\be
	\langle \tilde H_{\e}(t),\bvp \rangle_{\O}\leq C \|\nabla_x \bvp\|_{L^{2}({\O})}, \nn
	\ee
	meaning
	\be\label{H1}
	\tilde H_{\e} \in L^{\infty}(0,T;W^{-1,2}({\O})), \qquad \|\tilde H_\e\|_{L^\infty(0,T; W^{-1,2}(\O))} \leq C.
	\ee

	Combining the properties of the restriction operator in Lemma~\ref{Restriction operator}, we get moreover
	\ba
	\langle \tilde H_{\e},\bvp \rangle_{(0,T) \times \O}=\langle \nabla P_{\e},R_{\e}(\bvp) \rangle_{(0,T) \times \O_{\e}}=-\langle P_{\e},\dive R_{\e}(\bvp) \rangle_{(0,T) \times \O_{\e}}=0\nn
	\ea
	for all $\bvp\in C_{c}^{\infty}([0,T) \times \O;\R^{3})$ with $\dive \bvp=0$. Using Lemma~\ref{definition of P} and \eqref{H1}, we deduce that for any $t \in [0,T]$ there exists 
	\be
	\tilde P_{\e}(t) \in L_{0}^{2}({\O}) \nn
	\ee
	such that $\tilde H_{\e}(t) = \nabla_x \tilde P_{\e}(t)$ and
	\ba\label{pressure}
	\|\tilde P_{\e} \|_{L^\infty(0,T;L_{0}^{2}({\O}))}\leq\|\tilde H_{\e} \|_{L^\infty(0,T;W^{-1,2}({\O}))}\leq C. \nn
	\ea
	
	Finally, the very first property of $R_\e$ shows (see \cite{ALL-NS1}) that $\d_t \tilde P_\e |_{\O_\e} = p_\e$. 

\begin{remark}
Using the finer inequality \eqref{restOp1} instead of \eqref{restOp}, the same proof shows that indeed
\begin{gather*}
\tilde{P}_\e = \tilde{P}_\e^{(1)} + \s_\e \tilde{P}_\e^{(2)} \in L^\infty(0,T; W^{1,2}(\O) + L_0^2(\O)), \\
\|\tilde{P}_\e^{(1)}\|_{L^\infty(0,T; W^{1,2}(\O))} + \|\tilde{P}_\e^{(2)}\|_{L^\infty(0,T; L^2(\O))} \lesssim 1.
\end{gather*}
However, as we will see in the homogenization proof later on, it is enough for our purposes that $\tilde{P}_\e$ is uniformly bounded in $L^\infty(0,T; L^2(\O))$.
\end{remark}
	
	\section{Convergence and proof of Theorem~\ref{limitTorus}}\label{sec:conv}
	As already mentioned in the introduction, one of the main difficulties in proving convergence lies in the limit of the continuity equation, and more precisely, in finding the limit of the term $\vr_\e \uu_\e$ as both functions converge just weakly in their spaces. For this reason, to prove Theorems~\ref{limitTorus} and \ref{limitDomain}, we will make use of the relative energy inequality (REI), a brief derivation of which is given below. We will follow the presentation of \cite{Lu-Oschmann}.
	
	\subsection{Relative energy inequality}
	Our proof of convergence will rely on the widely used relative energy inequality (REI). This approach enables us additionally to give convergence rates for the velocity and density.\\
	
	Given a smooth function $\vU$ with $\vU |_{\d \O_\e} = 0$, the relative energy is defined as
	\begin{align*}
		E_\e(\vr_\e, \uu_\e | r, \vU) = \sigma_\e^4 \int_{\O_\e} \frac12 \vr_\e |\uu_\e - \vU|^2 \dx + \int_{\O_\e} \frac12 (\vr_\e - r)^2 \dx.
	\end{align*}
	Using moreover $\vU$ as test function in the momentum equation and noticing that $P_\e(0, \cdot) = 0$ gives
	\begin{align*}
		\sigma_\e^4 \left[ \int_{\O_\e} \vr_\e \uu_\e \cdot \vU \dx \right]_{t=0}^{t=\tau} &= \sigma_\e^4 \int_0^\tau \int_{\O_\e} \vr_\e \uu_\e \cdot \d_t \vU \dx \dt + \sigma_\e^4 \int_0^\tau \int_{\O_\e} \vr_\e \uu_\e \otimes \uu_\e : \nabla_x \vU \dx \dt \\
		&\quad - \mu \sigma_\e^2 \int_0^\tau \int_{\O_\e} \nabla_x \uu_\e : \nabla_x \vU \dx \dt + \int_0^\tau \int_{\O_\e} \vr_\e \vc f \cdot \vU \dx \dt \\
		&\quad + \int_0^\tau  \int_{\O_\e} P_\e \dive \d_t \vU \dx \dt - \int_{\O_\e} P_\e(\tau, \cdot) \dive \vU(\tau, \cdot) \dx.
	\end{align*}
	
	Further, using $\frac12 |\vU|^2$ in the continuity equation yields
	\begin{align*}
		\left[ \int_{\O_\e} \frac12 \vr_\e |\vU|^2 \dx \right]_{t=0}^{t=\tau} = \int_0^\tau \int_{\O_\e} \vr_\e \vU \cdot \d_t \vU \dx \dt + \int_0^\tau \int_{\O_\e} \vr_\e \uu_\e \otimes \vU : \nabla_x \vU \dx \dt.
	\end{align*}
	
	Recall also the energy inequality \eqref{enIneq} as
	\begin{align*}
		\sigma_\e^4 \left[ \int_{\O_\e} \frac12 \vr_\e |\uu_\e|^2 \dx \right]_{t=0}^{t=\tau} + \mu \s_\e^2 \int_0^\tau \int_{\O_\e} |\nabla_x \uu_\e|^2 \dx \dt \leq \int_0^\tau \int_{\O_\e} \vr_\e \vc f \cdot \uu_\e \dx \dt,
	\end{align*}
	which we can write as
	\begin{align*}
		&\sigma_\e^4 \left[ \int_{\O_\e} \frac12 \vr_\e |\uu_\e - \vU|^2 \dx \right]_{t=0}^{t=\tau} + \sigma_\e^4 \left[ \int_{\O_\e} \vr_\e \uu_\e \cdot \vU \dx \right]_{t=0}^{t=\tau} - \sigma_\e^4 \left[ \int_{\O_\e} \frac12 \vr_\e |\vU|^2 \right]_{t=0}^{t=\tau} + \mu \s_\e^2 \int_0^\tau \int_{\O_\e} |\nabla_x \uu_\e|^2 \dx \dt \\
		&\leq \int_0^\tau \int_{\O_\e} \vr_\e \vc f \cdot \uu_\e \dx \dt.
	\end{align*}
	
	Putting together, we find
	\begin{align*}
		&\sigma_\e^4 \left[ \int_{\O_\e} \frac12 \vr_\e |\uu_\e - \vU|^2 \dx \right]_{t=0}^{t=\tau} + \sigma_\e^4 \int_0^\tau \int_{\O_\e} \vr_\e \uu_\e \cdot \d_t \vU \dx \dt + \s_\e^4 \int_0^\tau \int_{\O_\e} \vr_\e \uu_\e \otimes \uu_\e : \nabla_x \vU \dx \dt \\
		&- \mu \s_\e^2 \int_0^\tau \int_{\O_\e} \nabla_x \uu_\e : \nabla_x \vU \dx \dt + \int_0^\tau \int_{\O_\e} \vr_\e \vc f \cdot \vU \dx \dt - \sigma_\e^4 \int_0^\tau \int_{\O_\e} \vr_\e \vU \cdot \d_t \vU \dx \dt \\
		&- \s_\e^4 \int_0^\tau \int_{\O_\e} \vr_\e \uu_\e \otimes \vU : \nabla_x \vU \dx \dt + \mu \sigma_\e^2 \int_0^\tau \int_{\O_\e} |\nabla_x \uu_\e|^2 \dx \dt \\
		&\leq \int_0^\tau \int_{\O_\e} \vr_\e \vc f \cdot \uu_\e \dx \dt + \int_0^\tau  \int_{\O_\e} P_\e \dive \d_t \vU \dx \dt - \int_{\O_\e} P_\e(\tau, \cdot) \dive \vU(\tau, \cdot) \dx,
	\end{align*}
	or in rearranged form
	\ba\label{rei1}
	&\sigma_\e^4 \left[ \int_{\O_\e} \frac12 \vr_\e |\uu_\e - \vU|^2 \dx \right]_{t=0}^{t=\tau} + \mu \s_\e^2 \int_0^\tau \int_{\O_\e} |\nabla_x (\uu_\e - \vU)|^2 \dx \dt \\
	&\leq - \sigma_\e^4 \int_0^\tau \int_{\O_\e} \vr_\e (\uu_\e - \vU) \cdot (\d_t \vU + (\uu_\e \cdot \nabla_x) \vU) \dx \dt + \int_0^\tau \int_{\O_\e} \vr_\e \vc f \cdot (\uu_\e - \vU) \dx \dt \\
	&\quad - \mu \s_\e^2 \int_0^\tau \int_{\O_\e} \nabla_x \vU : \nabla_x (\uu_\e - \vU) \dx \dt + \int_0^\tau  \int_{\O_\e} P_\e \dive \d_t \vU \dx \dt - \int_{\O_\e} P_\e(\tau, \cdot) \dive \vU(\tau, \cdot) \dx.
	\ea
	
	In order to take care of the density, we see that by conservation of $L^q$-norms \eqref{consLp}, we have for any $r \in C^\infty([0,T] \times \overline{\O})$ and with the help of the continuity equation \eqref{NSE}$_1$
	\begin{align*}
		0 &= \left[ \int_{\O_\e} \frac12 \vr_\e^2 \dx \right]_{t=0}^{t=\tau} = \left[ \int_{\O_\e} \frac12 (\vr_\e - r)^2 \dx \right]_{t=0}^{t=\tau} + \left[ \int_{\O_\e} r \vr_\e \dx \right]_{t=0}^{t=\tau} - \left[ \int_{\O_\e} \frac12 r^2 \dx \right]_{t=0}^{t=\tau} \\
		&= \left[ \int_{\O_\e} \frac12 (\vr_\e - r)^2 \dx \right]_{t=0}^{t=\tau} + \int_0^\tau \int_{\O_\e} \d_t( r \vr_\e ) \dx \dt - \int_0^\tau \int_{\O_\e} \frac12 \d_t(r^2) \dx \dt \\
		&= \left[ \int_{\O_\e} \frac12 (\vr_\e - r)^2 \dx \right]_{t=0}^{t=\tau} + \int_0^\tau \int_{\O_\e} \d_t r \vr_\e + r \d_t \vr_\e \dx \dt - \int_0^\tau \int_{\O_\e} r \d_t r \dx \dt \\
		&= \left[ \int_{\O_\e} \frac12 (\vr_\e - r)^2 \dx \right]_{t=0}^{t=\tau} + \int_0^\tau \int_{\O_\e} (\vr_\e - r) \d_t r - r \dive(\vr_\e \uu_\e) \dx \dt \\
		&= \left[ \int_{\O_\e} \frac12 (\vr_\e - r)^2 \dx \right]_{t=0}^{t=\tau} + \int_0^\tau \int_{\O_\e} (\vr_\e - r) \d_t r + \vr_\e \uu_\e \cdot \nabla_x r \dx \dt.
	\end{align*}
	
	Adding to \eqref{rei1}, we finally obtain
	\ba\label{rei}
	&\left[ E_\e(\vr_\e, \uu_\e | r, \vU) \right]_{t=0}^{t=\tau} + \mu \s_\e^2 \int_0^\tau \int_{\O_\e} |\nabla_x \uu_\e - \nabla_x \vU|^2 \dx \dt \\
	&\leq - \sigma_\e^4 \int_0^\tau \int_{\O_\e} \vr_\e (\uu_\e - \vU) \cdot (\d_t \vU + (\uu_\e \cdot \nabla_x) \vU) \dx \dt + \int_0^\tau \int_{\O_\e} \vr_\e \vc f \cdot (\uu_\e - \vU) \dx \dt \\
	&\quad - \mu \s_\e^2 \int_0^\tau \int_{\O_\e} \nabla_x \vU : \nabla_x (\uu_\e - \vU) \dx \dt + \int_0^\tau  \int_{\O_\e} P_\e \dive \d_t \vU \dx \dt - \int_{\O_\e} P_\e(\tau, \cdot) \dive \vU(\tau, \cdot) \dx \\
	&\quad - \int_0^\tau \int_{\O_\e} (\vr_\e - r) (\d_t r + \uu_\e \cdot \nabla_x r) \dx \dt,
	\ea
where we eventually used that $\int_{\O_\e} r \uu_\e \cdot \Grad r \dx = 0$ by solenoidality of $\uu_\e$. Note that this is particularly the REI for the compressible case, cf.~\cite{HNO24}. Such idea has been used recently to show a weak-strong uniqueness result for the inhomogeneous incompressible system, see \cite{CBSV24}­. Our proof of convergence is similar to \cite{HNO24}: we will choose $\vU = W_\e \uu$, where $\uu$ is a strong solution to Darcy's law, and $W_\e = (\vv_\e^1, \vv_\e^2, \vv_\e^3)$ is the capacity function for the holes constructed in Section~\ref{Local problem}. This procedure enables us to not only get convergence, but also rates.

\begin{remark}
By density, we can extend the class of test functions to $\vU, \d_t \vU \in L^q(0,T; W^{1,q}(\O_\e))$, $\vU = 0$ on $\d \O_\e$, and $r \in L^q(0,T; L^q(\O))$, for $q>1$ high enough such that all the integrals in \eqref{rei} are well-defined.
\end{remark}
	
	\subsection{Convergence for toroidal domain}
	Let $(\vr, \uu)$ be a strong solution to Darcy's law \eqref{Darcy} with the regularity required in Theorem~\ref{limitTorus}, and choose $\ww_\e = \vU = W_\e \uu$ and $r = \vr$ as test functions in the REI \eqref{rei}. Note in particular that this is possible since $\ww_\e |_{\d \O_\e} = 0$. This gives
	\begin{align*}
		&\left[ E_\e(\vr_\e, \uu_\e | \vr, \ww_\e) \right]_{t=0}^{t=\tau} + \mu \s_\e^2 \int_0^\tau \int_{\O_\e} |\nabla_x (\uu_\e - \ww_\e)|^2 \dx \dt \\
		&\leq - \sigma_\e^4 \int_0^\tau \int_{\O_\e} \vr_\e (\uu_\e - \ww_\e) \cdot (\d_t \ww_\e + (\uu_\e \cdot \nabla_x) \ww_\e) \dx \dt + \int_0^\tau \int_{\O_\e} \vr_\e \vc f \cdot (\uu_\e - \ww_\e) \dx \dt \\
		&\quad - \mu \s_\e^2 \int_0^\tau \int_{\O_\e} \nabla_x \ww_\e : \nabla_x (\uu_\e - \ww_\e) \dx \dt \\
		&\quad + \int_0^\tau  \int_{\O_\e} P_\e \dive \d_t \ww_\e \dx \dt - \int_{\O_\e} P_\e(\tau, \cdot) \dive \ww_\e(\tau, \cdot) \dx \\
		&\quad - \int_0^\tau \int_{\O_\e} (\vr_\e - \vr) ( \d_t \vr + \uu_\e \cdot \nabla_x \vr ) \dx \dt.
	\end{align*}
	
	As for the last two pressure terms, we deduce with \eqref{pressure} and $$\dive \ww_\e = W_\e : \nabla_x \uu = (W_\e - \Id) : \nabla_x \uu$$ by $\dive \uu = 0$ that
	\begin{align*}
		&\int_0^\tau  \int_{\O_\e} P_\e \dive \d_t \ww_\e \dx \dt - \int_{\O_\e} P_\e(\tau, \cdot) \dive \ww_\e(\tau, \cdot) \dx \leq C \|P_\e\|_{L^\infty(0,T; L^2(\O_\e))} \|\dive \ww_\e\|_{W^{1,2}(0,T; L^2(\O_\e))} \\
		&\leq C \|W_\e - \Id\|_{L^3(\O)} \|\Grad \uu\|_{W^{1,2}(0,T; L^6(\O))} \leq C \e^{\a- 1},
	\end{align*}
	where we used that $\uu \in L^2(0,T; W^{3,2}(\O))$ and $\d_t \uu \in L^2(0,T; W^{2,2}(\O))$, together with Sobolev embedding $W^{1,2}(\O) \hookrightarrow L^6(\O)$, and estimate \eqref{estVei}.\\
	
Next, since $(\vr, \uu)$ satisfies the continuity equation pointwise in $\O$ and therefore as well in $\O_\e$, we may write
\begin{align*}
&- \int_0^\tau \int_{\O_\e} (\vr_\e - \vr) ( \d_t \vr + \uu_\e \cdot \nabla_x \vr ) \dx \dt = - \int_0^\tau \int_{\O_\e} (\vr_\e - \vr) (\uu_\e - \uu) \cdot \nabla_x \vr \dx \dt \\
&= - \int_0^\tau \int_\O (\tilde{\vr}_\e - \vr) (\tilde{\uu}_\e - \uu) \cdot \nabla_x \vr \dx \dt + \int_0^\tau \int_{\O \setminus \O_\e} \vr \uu \cdot \nabla_x \vr \dx \dt.
\end{align*}
For the latter term, we use that $\ww_\e = 0$ in $\O \setminus \O_\e$ to write
\begin{align*}
\int_0^\tau \int_{\O \setminus \O_\e} \vr \uu \cdot \nabla_x \vr \dx \dt &= \int_0^\tau \int_{\O \setminus \O_\e} (\uu - \ww_\e) \cdot \nabla_x \frac12 \vr^2 \dx \dt \\
&\lesssim \| W_\e - \Id\|_{L^3(\O)} \|\uu \otimes \Grad \vr^2\|_{L^1(0,T; L^\frac32(\O))} \lesssim \e^{\a - 1}.
\end{align*}
 Moreover, we estimate
	\begin{align*}
		&\int_0^\tau \int_{\O} (\tilde \vr_\e - \vr) (\tilde \uu_\e - \uu) \cdot \nabla_x \vr \dx \dt \\
		& = \int_0^\tau \int_{\O} (\tilde \vr_\e - \vr) (\tilde \uu_\e - \ww_\e) \cdot \nabla_x \vr \dx \dt + \int_0^\tau \int_{\O} (\tilde \vr_\e - \vr) (\ww_\e - \uu) \cdot \nabla_x \vr \dx \dt \\
		&\leq C \int_0^\tau \int_{\O} |\tilde \vr_\e - \vr|^2 \dx \dt + \frac{\mu}{2} \s_\e^2 \int_0^\tau \int_{\O_\e} |\nabla_x (\uu_\e - \ww_\e)|^2 \dx \dt + C \|\ww_\e - \uu\|_{L^2((0,T) \times \O)}^2 \\
		&\leq C \int_0^\tau \int_{\O} |\tilde \vr_\e - \vr|^2 \dx \dt + \frac{\mu}{2} \s_\e^2 \int_0^\tau \int_{\O_\e} |\nabla_x (\uu_\e - \ww_\e)|^2 \dx \dt + C \e^{2(\a - 1)} \\
		&= C \int_0^\tau \int_{\O_\e} |\vr_\e - \vr|^2 \dx \dt + C \int_0^\tau \int_{\O \setminus \O_\e} \vr^2 \dx \dt + \frac{\mu}{2} \s_\e^2 \int_0^\tau \int_{\O_\e} |\nabla_x (\uu_\e - \ww_\e)|^2 \dx \dt + C \e^{2(\a - 1)} \\
		&= C \int_0^\tau \int_{\O_\e} |\vr_\e - \vr|^2 \dx \dt + \frac{\mu}{2} \s_\e^2 \int_0^\tau \int_{\O_\e} |\nabla_x (\uu_\e - \ww_\e)|^2 \dx \dt + C \big(\e^{3(\a - 1)} + \e^{2(\a - 1)} \big),
	\end{align*}
	where we used that $\ww_\e |_{\d \O_\e} = 0$, and estimate \eqref{estVei}. Further, we may absorb the second term by the left-hand side of the REI to obtain
	\begin{align*}
		&\left[ E_\e(\vr_\e, \uu_\e | \vr, \ww_\e) \right]_{t=0}^{t=\tau} + \frac{\mu}{2} \s_\e^2 \int_0^\tau \int_{\O_\e} |\nabla_x (\uu_\e - \ww_\e)|^2 \dx \dt \\
		&\leq - \sigma_\e^4 \int_0^\tau \int_{\O_\e} \vr_\e (\uu_\e - \ww_\e) \cdot (\d_t \ww_\e + (\uu_\e \cdot \nabla_x) \ww_\e) \dx \dt + \int_0^\tau \int_{\O_\e} \vr_\e \vc f \cdot (\uu_\e - \ww_\e) \dx \dt \\
		&\quad - \mu \s_\e^2 \int_0^\tau \int_{\O_\e} \nabla_x \ww_\e : \nabla_x (\uu_\e - \ww_\e) \dx \dt + C \int_0^\tau E_\e(\vr_\e, \uu_\e | \vr, \ww_\e) \dt + C \e^{\a - 1}.
	\end{align*}
	
	Using now Darcy's law \eqref{Darcy}, we replace $\vr_\e \vf$ to get
	\begin{align*}
		&\left[ E_\e(\uu_\e | \ww_\e) \right]_{t=0}^{t=\tau} + \frac{\mu}{2} \s_\e^2 \int_0^\tau \int_{\O_\e} |\nabla_x \uu_\e - \nabla_x \ww_\e|^2 \dx \dt \\
		&\leq - \s_\e^4 \int_0^\tau \int_{\O_\e} \vr_\e (\uu_\e - \ww_\e) \cdot \big( \d_t \ww_\e + (\uu_\e \cdot \nabla_x) \ww_\e \big) \dx \dt + \int_0^\tau \int_{\O_\e} (\mu M_0 \uu + \nabla_x p) \cdot (\uu_\e - \ww_\e) \dx \dt \\
		&\quad - \mu \s_\e^2 \int_0^\tau \int_{\O_\e} \nabla_x \ww_\e : \nabla_x (\uu_\e - \ww_\e) \dx \dt + C \int_0^\tau E_\e(\vr_\e, \uu_\e | \vr, \ww_\e) \dt \\
		&\quad + \int_0^\tau \int_{\O_\e} (\vr_\e - \vr) \vf \cdot (\uu_\e - \ww_\e) \dx \dt + C \e^{\a - 1}.
	\end{align*}
	
	The remaining force term we handle similar as before the density terms, giving rise to the same bounds.\\
	
	Next, noticing that $\dive \uu_\e = 0$ and $\uu_\e |_{\d \O_\e} = 0$, we get for the term involving $\nabla_x p$ together with the estimate \eqref{estVei}
	\begin{align*}
		&\left[ E_\e(\uu_\e | \ww_\e) \right]_{t=0}^{t=\tau} + \frac{\mu}{2} \s_\e^2 \int_0^\tau \int_{\O_\e} |\nabla_x (\uu_\e - \ww_\e)|^2 \dx \dt \\
		&\leq - \s_\e^4 \int_0^\tau \int_{\O_\e} \vr_\e (\uu_\e - \ww_\e) \cdot \big( \d_t \ww_\e + (\uu_\e \cdot \nabla_x) \ww_\e \big) \dx \dt + \int_0^\tau \int_{\O_\e} \mu M_0 \uu \cdot (\uu_\e - \ww_\e) \dx \dt \\
		&\quad - \mu \s_\e^2 \int_0^\tau \int_{\O_\e} \nabla_x \ww_\e : \nabla_x (\uu_\e - \ww_\e) \dx \dt + C \int_0^\tau E_\e(\vr_\e, \uu_\e | \vr, \ww_\e) \dt + C \e^{\a - 1}.
	\end{align*}
	
	For the first term on the RHS, we estimate with Young's and Poincar\'e's inequality \eqref{poinc} for $\delta>0$ small enough
	\begin{align*}
		&\s_\e^4 \int_0^\tau \int_{\O_\e} \vr_\e (\uu_\e - \ww_\e) \cdot \big( \d_t \ww_\e + (\uu_\e \cdot \nabla_x) \ww_\e \big) \dx \dt \\
		&= \s_\e^4 \int_0^\tau \int_{\O_\e} \vr_\e (\uu_\e - \ww_\e) \cdot \big( \d_t \ww_\e + (\ww_\e \cdot \nabla_x) \ww_\e \big) \dx \dt \\
		&\quad + \s_\e^4 \int_0^\tau \int_{\O_\e} \vr_\e (\uu_\e - \ww_\e) \otimes (\uu_\e - \ww_\e) : \nabla_x \ww_\e \dx \dt \\
		&\leq C \sigma_\e^4 \|\uu_\e - \ww_\e\|_{L^2((0,T) \times \O_\e)} + C \s_\e^4 \|\nabla_x \ww_\e\|_{L^2((0,T) \times \O_\e)} \|\uu_\e - \ww_\e\|_{L^2((0,T) \times \O_\e)} \\
		&\quad + C \s_\e^4 \|\uu_\e - \ww_\e\|_{L^6((0,T) \times \O_\e)}^2 \|\nabla_x \ww_\e\|_{L^\frac32((0,T) \times \O_\e)} \\
		&\leq C \sigma_\e^5 \|\nabla_x(\uu_\e - \ww_\e)\|_{L^2((0,T) \times \O_\e)} + C \s_\e^4 \|\nabla_x(\uu_\e - \ww_\e)\|_{L^2((0,T) \times \O_\e)} \\
		&\quad + C \sigma_\e^4 \e^{\a - 2} \|\nabla_x(\uu_\e - \ww_\e)\|_{L^2((0,T) \times \O_\e)}^2 \\
		&\leq C_\delta \s_\e^8 + \delta \s_\e^2 \|\nabla_x(\uu_\e - \ww_\e)\|_{L^2((0,T) \times \O_\e)}^2 + C_\delta \s_\e^6 + \delta \s_\e^2 \|\nabla_x(\uu_\e - \ww_\e)\|_{L^2((0,T) \times \O_\e)}^2 \\
		&\quad + C \sigma_\e^4 \e^{\a - 2} \|\nabla_x(\uu_\e - \ww_\e)\|_{L^2((0,T) \times \O_\e)}^2 \\
		&\leq C_\delta \sigma_\e^6 + (2\delta + C \s_\e^2 \e^{\a - 2}) \s_\e^2 \|\nabla_x(\uu_\e - \ww_\e)\|_{L^2((0,T) \times \O_\e)}^2,
	\end{align*}
	where we used the estimate \eqref{estGradv} for $\nabla_x \ww_\e$. Now, we have $C \s_\e^2 \e^{\a - 2} = C \e < \delta$ for $\e>0$ small enough; hence, for $\delta>0$ small enough, the last term can be absorbed by the left-hand side of the REI.\\
	%
	
	For the last remaining term, we write
	\begin{align*}
		-\mu \s_\e^2 \int_0^\tau \int_{\O_\e} \nabla_x \ww_\e : \nabla_x (\uu_\e - \ww_\e) \dx \dt &= \mu \s_\e^2 \int_0^\tau \int_{\O_\e} \big( \Delta_x (W_\e \uu) - \nabla_x (\vc q_\e \cdot \uu) \big) \cdot (\uu_\e - \ww_\e) \dx \dt \\
		&\quad+ \mu \s_\e^2 \int_0^\tau \int_{\O_\e} \nabla_x (\vc q_\e \cdot \uu) \cdot (\uu_\e - \ww_\e) \dx \dt \\
		&=: I_1 + I_2.
	\end{align*}
	
	For $I_2$, we find due to $\dive \uu_\e = \dive \uu = 0$, $\dive \ww_\e = (W_\e - \Id) : \nabla_x \uu$, and \eqref{estVei}
	\begin{align*}
		&- \mu \s_\e^2 \int_0^\tau \int_{\O_\e} \nabla_x (\vc q_\e \cdot \uu) \cdot (\uu_\e - \ww_\e) \dx \dt = \mu \s_\e^2 \int_0^\tau \int_{\O_\e} (\vc q_\e \cdot \uu) (W_\e - \Id) : \nabla_x \uu \dx \dt \\
		&\leq C \s_\e^2 \|\vc q_\e\|_{L^\frac32 ((0,T) \times \O_\e)} \|W_\e - \Id\|_{L^3((0,T) \times \O_\e)} \leq C \s_\e^2 \e^{\a - 2} \e^{\a - 1} = C \e^{\a}.
	\end{align*}
	
	For $I_1$, we rewrite exactly the same way as in \cite[Section~7.2]{Lu-Oschmann}
	\begin{align*}
		- \Delta_x(W_\e \uu) + \nabla_x(\vc q_\e \cdot \uu) = (-\Delta_x W_\e + \nabla_x \vc q_\e) \uu + \vc z_\e, \qquad \|\vc z_\e\|_{L^2((0,T) \times \O_\e)} \leq C \s_\e^{-1}.
	\end{align*}
	The very same arguments as in \cite{Lu-Oschmann} then give
	\begin{align*}
		&\mu \int_0^\tau \int_{\O_\e} \big( \s_\e^2 [ \Delta (W_\e \uu) - \nabla_x (\vc q_\e \cdot \uu) ] + M_0 \uu \big) \cdot (\uu_\e - \ww_\e) \dx \dt \\
		&\leq C_\delta \e^{\a - 1} + \delta \sigma_\e^2 \|\nabla_x (\uu_\e - \ww_\e)\|_{L^2((0,T) \times \O_\e)}^2,
	\end{align*}
	where again the last term can be absorbed by the left-hand side of the REI for $\delta>0$ small enough. In total, we arrive at
	\begin{align*}
		\left[ E_\e(\vr_\e, \uu_\e | \vr, \ww_\e) \right]_{t=0}^{t=\tau} + \s_\e^2 \int_0^\tau \int_{\O_\e} |\nabla_x ( \uu_\e - \ww_\e) |^2 \dx \dt \lesssim \int_0^\tau E_\e(\vr_\e, \uu_\e | \vr, \ww_\e) \dt + \e^{\a - 1} + \e^{3 - \a}.
	\end{align*}
	A Gr\"onwall type argument thus shows
	\begin{align*}
		\sup_{t \in [0, T]} E_\e(\vr_\e, \uu_\e | \vr, \vU) (t) \lesssim E_\e(\vr_\e, \uu_\e | \vr, \vU) (0) + \e^{\a - 1} + \e^{3 - \a}.
	\end{align*}
	Finally, we see
	\begin{align*}
		\|\tilde{\uu}_\e - \uu\|_{L^2((0,T) \times \O)}^2 &\lesssim \|\tilde{\uu}_\e - \ww_\e\|_{L^2((0,T) \times \O)}^2 + \|\ww_\e - \uu\|_{L^2((0,T) \times \O)}^2 \\
		&\lesssim \sigma_\e^2 \|\nabla_x (\tilde{\uu}_\e - \ww_\e) \|_{L^2((0,T) \times \O)}^2 + \|\ww_\e - \uu\|_{L^2((0,T) \times \O)}^2 \\
		&\lesssim E_\e(\vr_\e, \uu_\e | \vr, \vU) (0) + \e^{\a - 1} + \e^{3 - \a},
	\end{align*}
	and (see \cite{HNO24, Lu-Oschmann})
	\begin{align*}
		E_\e(\vr_\e, \uu_\e | \vr, \vU) (0) \lesssim \sigma_\e^4 \|\sqrt{\tilde \vr_\e^0} ( \tilde{\uu}_{\e}^0 - \uu(0, \cdot) )\|_{L^2(\O)}^2 + \|\tilde \vr_\e^0 - \vr(0, \cdot)\|_{L^2(\O)}^2 + \sigma_\e^4,
	\end{align*}
	giving the desired estimate
	\ba \label{finalREI}
	&\|\tilde{\uu}_\e - \uu\|_{L^2((0,T) \times \O)}^2 + \|\tilde \vr_\e - \vr\|_{L^2((0,T) \times \O)}^2 \\
	&\leq C \big( \sigma_\e^4 \|\sqrt{\tilde \vr_\e^0} ( \tilde{\uu}_{\e}^0 - \uu(0, \cdot) )\|_{L^2(\O)}^2 + \|\tilde \vr_\e^0 - \vr(0, \cdot)\|_{L^2(\O)}^2 + \e^{\a - 1} + \e^{3 - \a} \big).
	\ea
	This finishes the proof of Theorem~\ref{limitTorus}.
	
	\section{Adaptations for bounded domain, proof of Theorem~\ref{limitDomain}}\label{sec:bdDom}
	In a bounded domain $\O \subset \R^3$, we need to take care of the boundary mismatch between $\uu_\e|_{\d \O} = 0$ and $\uu \cdot \vc n|_{\d \O} = 0$, hence \emph{a priori} we don't know that $\ww_\e = 0$ on $\d \O$. Thus, we use the boundary corrector from \cite{Lu-Oschmann} (for previous results in this direction, see \cite{HNO24, Shen22}), the construction of which we briefly recall.
	
	Having $W_\e$ defined as in Proposition~\ref{TestFct}, we search for a solution to
	\begin{align*}
		\begin{cases}
			\mathbf{curl} \, \Phi_1^\e = W_\e(\e^\a \cdot) - \mathbb{I} & \text{in } \R^3,\\
			\Phi_1^\e \to 0 & \text{as } x \to \infty,
		\end{cases}
	\end{align*}
	where the curl of a matrix-valued function $A$ is taken row-wise, i.e., $({\rm curl} \, A)_{ij} = \epsilon_{ikl} \partial_k A_{jl}$, with $\epsilon_{ikl}$ being the Levi-Civita symbol\footnote{Letting $\delta_{ij}$ denote the usual Kronecker delta, we may write $\epsilon_{ikl} = \delta_{i1}(\delta_{k2} \delta_{l3} - \delta_{k3}\delta_{l2}) - \delta_{i2}(\delta_{k1} \delta_{l3} - \delta_{k3}\delta_{l1}) + \delta_{i3}(\delta_{k1} \delta_{l2} - \delta_{k2}\delta_{l1})$.}.\\
	
	As shown in \cite{Lu-Oschmann}, we may then define the actual corrector $\Phi^\e$ as $$ \Phi^\e(x) = \Phi_1^\e \left( \frac{x}{\e^\a} \right),$$ satisfying $$ \Phi^\e \in W^{1, \infty}(\R^3), \qquad \|\Phi^\e\|_{L^\infty(\R^3)} + \e^\a \|\nabla_x \Phi^\e\|_{L^\infty(\R^3)} \lesssim 1.$$
	
	Let now $\eta : [0, \infty) \to [0,1]$ be a smooth cut-off function with $\eta(0)=\eta'(0)=0$ and $\eta(x) = 1$ for $|x|>1$. Define
	\begin{align*}
		\eta_\e(x) = \eta \left( \frac{{\rm dist}(x, \partial \O)}{\e^\a} \right),
	\end{align*}
	and set finally
	\providecommand{\twe}{\tilde{\ww}_\e}
	\begin{align*}
		\twe = \e^\a \nabla_x \eta_\e \times (\Phi^\e \uu) + \eta_\e \vc w_\e = \e^\a \nabla_x \eta_\e \times (\Phi^\e \uu) + (\eta_\e - 1) (W_\e - \mathbb{I}) \uu + (\eta_\e - 1) \uu + \vc w_\e.
	\end{align*}
	
	Again as seen in \cite{Lu-Oschmann}, we have $\twe = 0$ on $\partial \O$, $\twe = \vc w_\e$ if ${\rm dist}(x, \partial \O) \geq \e^\a$, in particular $\twe = 0$ on $\O \setminus \O_\e$, and
	\begin{align*}
		\|\twe - \ww_\e\|_{L^\infty(0,T;L^q(\O))} + \e^\alpha \|\nabla_x (\twe - \ww_\e) \|_{L^\infty(0,T;L^q(\O))} &\lesssim \e^\frac{\a}{q}, \\
		\|\dive(\ww_\e - \twe)\|_{L^\infty(0,T; L^q(\O))} &\lesssim \e^{\min\{1, \frac3q\} (\a - 1)} + \e^\frac{\a}{q}.
	\end{align*}
	
	Inserting as before Darcy's law in the REI \eqref{rei}, we find
	\begin{align*}
		&\left[ E_\e(\uu_\e | \twe) \right]_{t=0}^{t=\tau} + \mu \s_\e^2 \int_0^\tau \int_{\O_\e} |\nabla_x (\uu_\e - \twe)|^2 \dx \dt \\
		&\leq - \s_\e^4 \int_0^\tau \int_{\O_\e} \vr_\e (\uu_\e - \twe) \cdot \big( \d_t \twe + (\uu_\e \cdot \nabla_x) \twe \big) \dx \dt + \int_0^\tau \int_{\O_\e} (\mu M_0 \uu + \nabla_x p) \cdot (\uu_\e - \twe) \dx \dt \\
		&\quad - \mu \s_\e^2 \int_0^\tau \int_{\O_\e} \nabla_x \twe : \nabla_x (\uu_\e - \twe) \dx \dt.
	\end{align*}
	
	The pressure term gives
	\begin{align*}
		\int_0^\tau \int_{\O_\e} \nabla_x p \cdot (\uu_\e - \twe) \dx \dt \lesssim \|\dive(\twe - \ww_\e)\|_{L^1((0,T) \times \O)} + \|\dive \ww_\e\|_{L^1((0,T) \times \O)} \lesssim \e^{\a - 1} + \e^\a.
	\end{align*}
	
	As for the convective one,
	\begin{align*}
		&\s_\e^4 \int_0^\tau \int_{\O_\e} \vr_\e (\uu_\e - \twe) \cdot \big( \d_t \twe + (\uu_\e \cdot \nabla_x) \twe \big) \dx \dt \\
		&\leq C \sigma_\e^4 \|\uu_\e - \twe\|_{L^2((0,T) \times \O_\e)} + C \s_\e^4 \|\nabla_x \twe\|_{L^2((0,T) \times \O_\e)} \|\uu_\e - \twe\|_{L^2((0,T) \times \O_\e)} \\
		&\quad + C \s_\e^4 \|\uu_\e - \twe\|_{L^6((0,T) \times \O_\e)}^2 \|\nabla_x \twe\|_{L^\frac32((0,T) \times \O_\e)} \\
		&\leq C \sigma_\e^4 \|\uu_\e - \twe\|_{L^2((0,T) \times \O_\e)} + C \s_\e^4 \|\nabla_x \ww_\e\|_{L^2((0,T) \times \O_\e)} \|\uu_\e - \twe\|_{L^2((0,T) \times \O_\e)} \\
		&\quad + C \s_\e^4 \|\nabla_x (\twe - \ww_\e)\|_{L^2((0,T) \times \O_\e)} \|\uu_\e - \twe\|_{L^2((0,T) \times \O_\e)} \\
		&\quad + C \s_\e^4 \|\uu_\e - \twe\|_{L^6((0,T) \times \O_\e)}^2 \|\nabla_x \ww_\e\|_{L^\frac32((0,T) \times \O_\e)} \\
		&\quad + C \s_\e^4 \|\uu_\e - \twe\|_{L^6((0,T) \times \O_\e)}^2 \|\nabla_x (\twe - \ww_\e)\|_{L^\frac32((0,T) \times \O_\e)} \\
		&\leq C_\delta \sigma_\e^6 + (2\delta + \e + C \s_\e^2 \e^{\frac{2\a}{3} - \a}) \s_\e^2 \|\nabla_x(\uu_\e - \twe)\|_{L^2((0,T) \times \O_\e)}^2 \\
		&\quad + C \s_\e^5 \e^{\frac{\a}{2} - \a} \|\nabla_x (\uu_\e - \twe)\|_{L^2((0,T) \times \O_\e)} \\
		&\leq C_\delta \sigma_\e^6 (1 + \s_\e^2 \e^{-\a}) + (3\delta + \e + C \s_\e^2 \e^{-\frac{\a}{3}}) \s_\e^2 \|\nabla_x(\uu_\e - \twe)\|_{L^2((0,T) \times \O_\e)}^2.
	\end{align*}
	
	Note that $\s_\e^8 \e^{-\a} = \e^{12-5\a} \to 0$ if $1 < \a < \frac{12}{5}$, and also $\s_\e^2 \e^{-\frac{\a}{3}} = \e^{3-\frac{4\a}{3}} < \delta$ just for $$ 1 < \a < \frac94 < \frac{12}{5}.$$
	
	For the stress term, we write
	\begin{align*}
		\mu \s_\e^2 \int_0^\tau \int_{\O_\e} \nabla_x \twe : \nabla_x (\uu_\e - \twe) \dx \dt &= \mu \s_\e^2 \int_0^\tau \int_{\O_\e} (- \Delta_x W_\e \uu - \vc z_\e) \cdot (\uu_\e - \twe) \dx \dt \\
		&\quad + \mu \s_\e^2 \int_0^\tau \int_{\O_\e} \nabla_x (\twe - \ww_\e) : \nabla_x (\uu_\e - \twe) \dx \dt.
	\end{align*}
	
	Hence, the only remaining term to estimate is
	\begin{align*}
		&\mu \s_\e^2 \int_0^\tau \int_{\O_\e} \nabla_x (\twe - \ww_\e) : \nabla_x (\uu_\e - \twe) \dx \dt \\
		&\leq C \s_\e^2 \|\nabla_x(\twe - \ww_\e)\|_{L^2((0,T) \times \O_\e)} \|\nabla_x (\uu_\e - \twe)\|_{L^2((0,T) \times \O_\e)} \\
		&\leq C_\delta \s_\e^2 \e^{-\a} + \delta \s_\e^2 \|\nabla_x(\uu_\e - \twe)\|_{L^2((0,T) \times \O_\e)}^2,
	\end{align*}
	and in turn we need $\s_\e^2 \e^{-\a} = \e^{3 - 2\a} \to 0$, leading to $\a < \frac32$.
	
	The final energy estimate is then
	\begin{align*}
		&\|\tilde{\uu}_\e - \uu\|_{L^2((0,T) \times \O)}^2 + \|\tilde \vr_\e - \vr\|_{L^2((0,T) \times \O)}^2 \\
		&\lesssim \|\sqrt{\tilde \vr_\e^0} (\tilde \uu_\e^0 - \uu(0, \cdot))\|_{L^2(\O)}^2 + \|\tilde \vr_\e^0 - \vr(0, \cdot)\|_{L^2(\O)}^2 + \e^{\a - 1} + \e^{3 - \a} + \e^{3 - 2\a} \\
		&\lesssim \|\sqrt{\tilde \vr_\e^0} (\tilde \uu_\e^0 - \uu(0, \cdot))\|_{L^2(\O)}^2 + \|\tilde \vr_\e^0 - \vr(0, \cdot)\|_{L^2(\O)}^2 + \e^{\a - 1} + \e^{3 - 2\a},
	\end{align*}
	thus finishing the proof of Theorem~\ref{limitDomain}.
	
	\begin{remark}
		We shall compare the requirement $1 < \a < \frac32$ with the one made in \cite{Lu-Oschmann}: in there, $\s_\e^4$ is replaced by a more general $\e^\l$ and the assumption $\lambda > \a$ is made. Going to our setting where $\lambda = 2(3-\a)$, this means $\a < 2$; it is thus expected that the above results holds in the range $\a \in (1, 2)$ with a possibly different error term. Moreover, this hints that the error estimates given in \cite{ALL-NS2} also hold just for this range (in a bounded domain), although the \emph{qualitative} convergence result holds for all $1 <\a < 3$.
	\end{remark}
	
	\section{Pressure convergence}\label{sec:pressConv}
	In this section, we will show convergence rates for the pressure in the case of $\O = \mathbb{T}^3$. The estimates for a bounded domain $\O \subset \R^3$ then follow with obvious changes. More precisely, we aim to prove:
	\begin{theorem}\label{pressRate}
		Let $\tilde P_\e$ be the pressure extension from Section~\ref{sec:pressExt}, and let $\tilde p_\e = \d_t \tilde P_\e$. Then, under the assumptions of Theorem~\ref{limitTorus}, we have
		\begin{align*}
			&\|\tilde p_\e - p\|_{W^{-1,2}(0,T; L_0^2(\O))} \leq C (\s_\e^2 \|\sqrt{\tilde \vr_\e^0} (\tilde \uu_\e^0 - \uu(0, \cdot))\|_{L^2(\O)} + \|\tilde \vr_\e^0 - \vr^0\|_{L^2(\O)} + \e^\frac{\a - 1}{2} + \s_\e).
		\end{align*}
	\end{theorem}

	\begin{proof}
		\providecommand{\vF}{\vc{F}}
		Recall from Section~\ref{sec:pressExt} that $\tilde P_\e \in L^\infty(0,T;L_0^2(\O))$ is uniformly bounded. Let $p$ be the pressure of Darcy's law \eqref{Darcy}, and define $\Pi_\e$ by
		\begin{align*}
			\langle \nabla_x \Pi_\e, \vv \rangle_{\O} = \langle \nabla_x P, R_\e(\vv) \rangle_{\O_\e} \mbox{ for any } \vv \in L^2(0,T; W^{1,2}(\Omega)), \qquad P = \int_0^t p(s, \cdot) \, {\rm d} s.
		\end{align*}
		Similar arguments as before and the strong solution property of $(\vr, \uu)$ show $\Pi_\e \in L^\infty(0,T; L_0^2(\O))$, and $\d_t \Pi_\e |_{\O_\e} = p$. In other words, $\Pi_\e$ is the extension of $1_{\O_\e} P$, defined by duality with the help of the operator $R_\e$. We will show that $\nabla_x (\tilde P_\e - \Pi_\e) \to 0$ strongly in $L^2(0,T; W^{-1,2}(\O))$. This is essentially the same procedure as in \cite{ALL-NS1, ALL-NS2}.\\
		
		To this end, let $\vv \in L^2(0,T; W^{1,2}(\O))$ and use that by definition
		\begin{align}\label{dva}
			\langle \nabla_x \tilde P_\e - \nabla_x \Pi_\e, \vv \rangle_{\O} = \langle \nabla_x P_\e - \nabla_x P, R_\e(\vv) \rangle_{\O_\e}.
		\end{align}
		
		Now, set
		\begin{align*}
			\vU = \int_0^t \uu(s, \cdot) \, {\rm d} s, && \vF = \int_0^t (\vr \vf)(s, \cdot) \, {\rm d} s, && P = \int_0^t p(s, \cdot) \, {\rm d} s,
		\end{align*}
		where $(\vr, \uu, p)$ is the strong solution to Darcy's law \eqref{Darcy}. Then, since $M_0$ is constant, for any $t \in (0,T)$
		\begin{align*}
			\mu M_0 \vU = \vF - \nabla_x P \mbox{ in } \mathcal{D}'(\O);
		\end{align*}
		since we have actually a strong solution, this also holds pointwise, so especially pointwise in $\O_\e$.\\
		
		To shorten the notation, we will write just $\langle \cdot , \cdot \rangle$ instead of $\langle \cdot , \cdot \rangle_{\O / \O_\e}$ in what follows. Replacing the gradients in \eqref{dva} and integrating by parts, we get
		\begin{align*}
			\langle \nabla_x \tilde P_\e - \nabla_x \Pi_\e, \vv \rangle &= \langle \mathbf{F}_\e- \s_\e^4 (\vr_\e \uu_\e - \vr_\e^0 \uu_{\e}^0) - \s_\e^4 \dive \Psi_\e + \s_\e^2 \mu \Delta_x \vU_\e, R_\e(\vv) \rangle - \langle \vF - \mu M_0 \vU, R_\e(\vv) \rangle \\
			&= \langle \vF_\e - \vF, R_\e(\vv) \rangle - \s_\e^4 \langle \uu_\e-\uu_{\e0}, R_\e(\vv) \rangle + \s_\e^4 \langle \Psi_\e, \nabla_x R_\e(\vv) \rangle \\
			&\quad - \s_\e^2 \mu \langle \nabla_x \vU_\e, \nabla_x R_\e(\vv) \rangle + \mu \langle M_0 \vU, R_\e(\vv) \rangle.
		\end{align*}
		
		The first three terms are handled as follows: first,
		\begin{align*}
			| \langle \vF_\e - \vF, R_\e(\vv) \rangle | &= \int_0^\tau \int_{\O_\e} (\vr_\e - \vr) \vf \cdot R_\e(\vv) \dx \dt \leq C \|\vr_\e - \vr\|_{L^2((0,T) \times \O_\e)} \|R_\e\vv\|_{L^2((0,T) \times \O_\e)} \\
			&\leq C \|\vr_\e - \vr\|_{L^2((0,T) \times \O_\e)} \|\nabla_x \vv\|_{L^2((0,T) \times \O)} \\
			&\leq C \big( \sigma_\e^2 \|\sqrt{\tilde \vr_\e^0} ( \tilde{\uu}_{\e}^0 - \uu(0, \cdot) )\|_{L^2(\O)} + \|\tilde \vr_\e^0 - \vr(0, \cdot)\|_{L^2(\O)} + \e^\frac{\a - 1}{2} + \s_\e \big) \|\nabla_x \vv\|_{L^2((0,T) \times \O)},
		\end{align*}
		where we used the bounds on $(\vr_\e - \vr)$ already obtained in \eqref{finalREI}.\\
		
		Next,
		\begin{align*}
			| \s_\e^4 \langle \vr_\e \uu_\e, R_\e(\vv) \rangle | &\leq C \s_\e^4 \|\sqrt{\vr_\e}\|_{L^\infty(0,T; L^\infty(\O_\e))} \|\sqrt{\vr_\e} \uu_\e\|_{L^\infty(0,T; L^2(\O_\e))} \|R_\e(\vv)\|_{L^2((0,T) \times \O_\e)} \\
			&\leq C \s_\e^2 \|\nabla_x \vv\|_{L^2((0,T) \times \O_\e)},
		\end{align*}
		and
		\begin{align*}
			| \s_\e^4 \langle \vr_\e^0 \uu_\e^0, R_\e(\vv) \rangle | \leq C \s_\e^4 \|\sqrt{\vr_\e^0}\|_{L^\infty(\O_\e)} \|\sqrt{\vr_\e^0} \uu_\e^0\|_{L^2(\O_\e)} \|R_\e(\vv)\|_{L^2((0,T) \times \O_\e} \leq C \s_\e^2 \|\nabla_x \vv\|_{L^2((0,T) \times \O_\e)}.
		\end{align*}
		
		Moreover, we get
		\begin{align*}
			| \s_\e^4 \langle \Psi_\e, \nabla_x R_\e(\vv) \rangle | &\leq C \s_\e^4 \|\Psi_\e\|_{L^\infty(0,T; L^2(\O_\e))} \|\nabla_x R_\e(\vv)\|_{L^2((0,T) \times \O_\e)} \\
			&\leq C \s_\e^4 \|\vr_\e\|_{L^\infty(0,T; L^\infty(\O_\e))} \|\uu_\e\|_{L^2(0,T; L^4(\O_\e))}^2 \|\nabla_x R_\e(\vv)\|_{L^2((0,T) \times \O_\e)} \\
			&\leq C \s_\e \|\nabla_x \vv\|_{L^2((0,T) \times \O)}.
		\end{align*}
		
		For the remaining term, we write
		\begin{align*}
			\s_\e^2 \mu \langle \nabla_x \vU_\e, \nabla_x R_\e(\vv) \rangle = \s_\e^2 \mu \langle \nabla_x (\vU_\e - W_\e \uu), \nabla_x R_\e(\vv) \rangle + \s_\e^2 \mu \langle \nabla_x (W_\e \uu), \nabla_x R_\e(\vv) \rangle.
		\end{align*}
		Now, from the error estimates \eqref{finalREI} we infer
		\begin{align*}
			&| \s_\e^2 \mu \langle \nabla_x (\vU_\e - W_\e \uu), \nabla_x R_\e(\vv) \rangle | \leq C \s_\e^2 \|\nabla_x (\vU_\e - W_\e \uu)\|_{L^2((0,T) \times \O_\e)} \|\nabla_x R_\e(\vv)\|_{L^2((0,T) \times \O_\e)} \\
			&\leq C (\s_\e^2 \|\sqrt{\tilde \vr_\e^0} (\tilde \uu_\e^0 - \uu(0, \cdot))\|_{L^2(\O)} + \|\tilde \vr_\e^0 - \vr^0\|_{L^2(\O)} + \e^\frac{\a - 1}{2} + \s_\e) \|\nabla_x \vv\|_{L^2((0,T) \times \O)}.
		\end{align*}
		Hence, it remains to estimate
		\begin{align*}
			\mu \langle M_0 \vU, R_\e(\vv) \rangle - \s_\e^2 \mu \langle \nabla_x (W_\e \uu), \nabla_x R_\e(\vv) \rangle .
		\end{align*}
		But integration by parts and the very same arguments as before show
		\begin{align*}
			| \mu \langle M_0 \vU, R_\e(\vv) \rangle - \s_\e^2 \mu \langle \nabla_x (W_\e \uu), \nabla_x R_\e(\vv) \rangle | \leq C(\e^\frac{\a - 1}{2} + \s_\e) + | \mu \s_\e^2 \langle (\vc q_\e \cdot \nabla_x) \uu, R_\e(\vv) \rangle |,
		\end{align*}
		and eventually,
		\begin{align*}
			| \mu \s_\e^2 \langle (\vc q_\e \cdot \nabla_x) \uu, R_\e(\vv) \rangle | &\leq C \s_\e^2 \|\vc q_\e\|_{W^{-1,2}(\O)} \|\nabla_x \uu\|_{W^{1, \infty}(\O)} \|R_\e(\vv)\|_{L^2(0,T; W_0^{1,2}(\O_\e))} \\
			&\leq C \e \s_\e \|\nabla_x \vv\|_{L^2((0,T) \times \O)},
		\end{align*}
		where we used (see \cite[Lemma~2.4.1]{ALL-NS1}) $$ \|\vc q_\e\|_{W^{-1,2}(\O)} \leq C \e.$$
		
		Finally, we put together the above estimates to conclude
		\begin{align*}
			\|\nabla_x(\tilde P_\e - \Pi_\e)\|_{L^2(0,T; W^{-1,2}(\O))} \leq C (\s_\e^2 \|\sqrt{\tilde \vr_\e^0} (\tilde \uu_\e^0 - \uu(0, \cdot))\|_{L^2(\O)} + \|\tilde \vr_\e^0 - \vr^0\|_{L^2(\O)} + \e^\frac{\a - 1}{2} + \s_\e).
		\end{align*}
		
		Since $\tilde P_\e$ and $\Pi_\e$ have zero mean value over $\O$, Poincar\'{e}-Wirtinger inequality also gives
		\begin{align*}
			\|\tilde P_\e - \Pi_\e\|_{L^2(0,T; L_0^2(\O))} \leq C (\s_\e^2 \|\sqrt{\tilde \vr_\e^0} (\tilde \uu_\e^0 - \uu(0, \cdot))\|_{L^2(\O)} + \|\tilde \vr_\e^0 - \vr^0\|_{L^2(\O)} + \e^\frac{\a - 1}{2} + \s_\e).
		\end{align*}
		
		Furthermore, we can use $\d_t (\tilde P_\e , \Pi_\e) |_{\d \O_\e} = (p_\e, 1_{\O_\e} p)$ and continuity of $\d_t$ from $L^2(0,T)$ to $W^{-1,2}(0,T)$ to deduce
		\begin{align*}
			\|\tilde p_\e - p\|_{W^{-1,2}(0,T; L_0^2(\O))} &\leq \|p_\e - p\|_{W^{-1,2}(0,T; L_0^2(\O_\e))} + \|p\|_{W^{-1,2}(0,T; L_0^2(\O \setminus \O_\e))} \\
			&\leq \|p_\e - p\|_{W^{-1,2}(0,T; L_0^2(\O_\e))} + C \e^{3(\a - 1)} \\
			&= \|\d_t (\tilde P_\e - \Pi_\e)\|_{W^{-1,2}(0,T; L_0^2(\O_\e))} + C \e^{3(\a - 1)} \\
			&\leq C \|\tilde P_\e - \Pi_\e\|_{L^2(0,T; L_0^2(\O))} + C \e^{3(\a - 1)} \\
			&\leq C (\s_\e^2 \|\sqrt{\tilde \vr_\e^0} (\tilde \uu_\e^0 - \uu(0, \cdot))\|_{L^2(\O)} + \|\tilde \vr_\e^0 - \vr^0\|_{L^2(\O)} + \e^\frac{\a - 1}{2} + \s_\e).
		\end{align*}
		
		Note that this is precisely the convergence rate from \cite{ALL-NS2} for the stationary setting. This finishes the proof of Theorem~\ref{pressRate}.
	\end{proof}
	
	\begin{remark}
		Choosing in the above estimates the function $\vv \in W^{1,2}(\O)$ to be independent of time, we can show by the same arguments that
		\begin{align*}
			\|\tilde P_\e - \Pi_\e\|_{[L^\infty(0,T; L^2(\O))]'} \leq C (\s_\e^2 \|\sqrt{\tilde \vr_\e^0} (\tilde \uu_\e^0 - \uu(0, \cdot))\|_{L^2(\O)} + \|\tilde \vr_\e^0 - \vr^0\|_{L^2(\O)} + \e^\frac{\a - 1}{2} + \s_\e).
		\end{align*}
		This of course is also a consequence of the embedding $L^2 = [L^2]' \hookrightarrow [L^\infty]'$.
	\end{remark}
	
	\begin{remark}
		For a bounded smooth domain $\Omega \subset \R^3$, the same arguments as before give rise to
		\begin{align*}
			\|\tilde p_\e - p\|_{W^{-1,2}(0,T; L_0^2(\O))} \leq C (\s_\e^2 \|\sqrt{\tilde \vr_\e^0} (\tilde \uu_\e^0 - \uu(0, \cdot))\|_{L^2(\O)} + \|\tilde \vr_\e^0 - \vr^0\|_{L^2(\O)} + \e^\frac{\a - 1}{2} + \e^{\frac{3-2\a}{2}}).
		\end{align*}
	\end{remark}
	
	\section{Existence of strong solutions for the target system} \label{sec:ExistenceStrong}
	
	In this last section, we aim to prove the following result. 
	
	\begin{theorem} \label{Local Existence}
		Let $\Omega \subset \mathbb{R}^3$ be a bounded domain of class $C^3$, and let $\vr_0 \in W^{3,2}(\Omega)$, $\vr_0 \geq 0$, and $\vf$ such that
		\begin{equation*}
			\vf \in L^\infty(0,T; W^{3,2}(\Omega; \mathbb{R}^3)), \quad \partial_t \vf \in L^2(0,T; W^{2,2}(\Omega; \mathbb{R}^3))
		\end{equation*}
		be given functions. 
		
		Then, there exists a positive time $T^*$ such that problem \eqref{Darcy} admits a unique solution $(\vr, \vu)$ in $(0,T^*) \times \Omega$ such that $\vr(0,\cdot)=\vr_0$, $\vr\geq 0$ in $(0,T^*) \times \Omega$, and 
		\begin{align*}
			\vr \in C([0,T^*]; W^{3,2}(\Omega)), &\quad \partial_t \vr \in L^2(0,T^*; W^{2,2}(\Omega)), \\
			\vu \in L^2(0,T^*; W^{3,2}(\Omega)), &\quad \partial_t \vu \in L^2(0,T^*; W^{2,2}(\Omega)). 
		\end{align*}
	\end{theorem}
	\begin{remark}
		We point out that, by the Sobolev embedding $W^{1,2}(0,T) \hookrightarrow C([0,T])$, the solution $\vu$ of Theorem~\ref{Local Existence} belongs in particular to the class $$\vu \in C([0,T]; W^{2,2}(\Omega)).$$
	\end{remark}
	
	\begin{remark}
	The torus case can be proven the same way as shown below; the interested reader may consult \cite{Basaric2023error} as well.
	\end{remark}

	In order to prove Theorem~\ref{Local Existence}, we notice that, introducing the inverse of the resistance matrix $A= (\mu M_0)^{-1}$, system \eqref{Darcy} can be reformulated as 
	\begin{equation} \label{Darcy2}
		\begin{cases}
			\partial_t \vr + \vu \cdot \Grad \vr =0  &\mbox{in } (0,T) \times \Omega, \\
			\vu = \vr A \vf - A \Grad p &\mbox{in } (0,T) \times \Omega, \\
			\dive (A \Grad p) = \dive (\vr A\vf ) &\mbox{in } (0,T) \times \Omega, \\
			(A \Grad p) \cdot \textbf{n} = (\vr A \vf) \cdot \textbf{n} &\mbox{on } (0,T) \times \partial \Omega.
		\end{cases}
	\end{equation} 

	First of all, we focus on the solvability of the transport equation. 
	\begin{lemma}
		For any fixed $	\widetilde{\vu} \in L^2(0,T; W^{3,2}(\Omega))$ such that $\dive \widetilde{\vu}=0$, $\widetilde{\vu} \cdot \textbf{\textup{n}}|_{[0,T] \times \partial \Omega}=0$, and any $\vr_0 \in W^{3,2}(\Omega)$, $\vr_0\geq 0$, the transport equation 
		\begin{equation} \label{transport equation}
			\partial_t \vr + \widetilde{\vu} \cdot \Grad \vr =0 
		\end{equation}
		admits a unique solution $\vr \in C([0,T]; W^{3,2}(\Omega))$ such that $\vr \geq 0$ in $(0,T) \times \Omega$; in particular, the following estimate holds: 
		\begin{equation} \label{estimate density}
			\begin{aligned}
				&\sup_{t \in [0, T]} \| \vr(t, \cdot ) \|_{W^{3,2}(\Omega)}^2 \\
				&\leq \| \vr_0\|_{W^{3,2}(\Omega)}^2 \left(1+ C_1  \int_{0}^{T} \left[\|  \widetilde{\vu}(t, \cdot) \|_{W^{3,2}(\Omega)}^2 + 1 \right]  \dt \right) \exp \left(C_1\int_{0}^{T} \left[ \|  \widetilde{\vu}(t, \cdot) \|_{W^{3,2}(\Omega)}^2 + 1 \right] \dt \right).  
			\end{aligned}
		\end{equation}
	\end{lemma}
		\begin{proof}
		The existence and uniqueness of $\vr$ can be easily proven by the method of characteristics, as well its non-negativity in $(0,T) \times \Omega$. More precisely, defining for any $(t,x) \in [0,T] \times \Omega$ the flow map $X(s; t,x) = \gamma(s)$ with
\begin{align*}
\begin{cases}
\dot{\gamma}(s) = \tilde{\uu}(s, \gamma(s)),\\
\gamma(t) = x,
\end{cases}
\end{align*}
we find that $\vr(t,x) = \vr_0(X(0; t,x))$ is the unique solution to \eqref{transport equation}.
	
To get estimate \eqref{estimate density}, we first multiply equation \eqref{transport equation} by $\vr$, integrate over $\Omega$, and perform an integration by parts to end up with 
		\begin{equation} \label{e1}
			\frac12 \frac{\textup{d}}{\dt} \| \vr \|_{L^2(\Omega)}^2 =0. 
		\end{equation}
		Secondly, we take the gradient of \eqref{transport equation}, multiply the resulting expression by $\Grad \vr $, and integrate over $\Omega$; integrating by parts the term 
		\begin{equation*}
			\int_{\Omega} \widetilde{\vu} \cdot \Grad^2 \vr  \cdot \Grad \vr \ \dx = \frac{1}{2} \int_{\Omega} \widetilde{\vu} \cdot \Grad |\Grad \vr|^2 \ \dx =0, 
		\end{equation*}
		and using the Sobolev embedding $W^{2,2}(\Omega) \hookrightarrow L^{\infty}(\Omega)$, we get the estimate 
		\begin{equation} \label{e2}
			\frac12 \frac{\textup{d}}{\dt} \| \Grad \vr \|_{L^2(\Omega)}^2 \leq \| \Grad \widetilde{\vu} \|_{L^{\infty}(\Omega)} \| \Grad \vr \|_{L^2(\Omega)}^2 \lesssim \|  \widetilde{\vu} \|_{W^{3,2}(\Omega)} \| \vr \|_{W^{1,2}(\Omega)}^2. 
		\end{equation}
		Proceeding in a similar way for the second $D_x^2$ and third $D_x^3$ derivatives, and using the Sobolev embedding $W^{1,2}(\Omega) \hookrightarrow L^{6}(\Omega)$, we end up with 
		\begin{align}
			\frac12 \frac{\textup{d}}{\dt} \| D_x^2\vr \|_{L^2(\Omega)}^2 &\lesssim \| D_x^2 \widetilde{\vu} \|_{L^6(\Omega)} \| \Grad  \vr \|_{L^6(\Omega)} \| D_x^2 \vr \|_{L^2(\Omega)} + \| \Grad \widetilde{\vu}\|_{L^{\infty(\Omega)}} \| D_x^2 \vr \|_{L^2(\Omega)}^2 \notag \\
			&\lesssim \|  \widetilde{\vu} \|_{W^{3,2}(\Omega)} \| \vr \|_{W^{2,2}(\Omega)}^2, \label{e3}\\
			\frac12 \frac{\textup{d}}{\dt} \| D_x^3\vr \|_{L^2(\Omega)}^2 &\lesssim \| D_x^3 \widetilde{\vu} \|_{L^2(\Omega)}^2 \| \Grad \vr \|_{L^{\infty}(\Omega)}^2 + \| D_x^3 \vr \|_{L^2(\Omega)}^2 \notag \\
			&+ \| D_x^2 \widetilde{\vu} \|_{L^6(\Omega)} \| D_x^2  \vr \|_{L^6(\Omega)} \| D_x^3 \vr \|_{L^2(\Omega)} + \| \Grad \widetilde{\vu}\|_{L^{\infty}(\Omega)} \| D_x^3 \vr \|_{L^2(\Omega)}^2 \notag \\
			&\lesssim \left( \|  \widetilde{\vu} \|_{W^{3,2}(\Omega)}^2 + 1 \right)  \| \vr \|_{W^{3,2}(\Omega)}^2. \label{e4}
		\end{align}
		Summing up estimates \eqref{e1}--\eqref{e4} and integrating the resulting expression over time, it follows that for any $t \in [0,T]$ 
		\begin{equation*}
			\| \vr(t, \cdot ) \|_{W^{3,2}(\Omega)}^2 \leq \| \vr_0\|_{W^{3,2}(\Omega)}^2  + C_1 \int_{0}^{t} \left( \|  \widetilde{\vu}(s, \cdot) \|_{W^{3,2}(\Omega)}^2 + 1 \right)  \| \vr(s,\cdot) \|_{W^{3,2}(\Omega)}^2 \textup{d}s. 
		\end{equation*}
		By Gr\"onwall's lemma we obtain
		\begin{align*}
			\| &\vr(t, \cdot ) \|_{W^{3,2}(\Omega)}^2 \\
			&\leq \| \vr_0\|_{W^{3,2}(\Omega)}^2 \left(1+ C_1 \int_{0}^{t} \left[\|  \widetilde{\vu}(s, \cdot) \|_{W^{3,2}(\Omega)}^2 + 1 \right] \exp \left(C_1\int_{s}^{t} \left[ \|  \widetilde{\vu}(r, \cdot) \|_{W^{3,2}(\Omega)}^2 + 1 \right] \textup{d}r \right) \textup{d}s \right) \\
			&\leq \| \vr_0\|_{W^{3,2}(\Omega)}^2 \left(1+ C_1 \exp \left(C_1\int_{0}^{t} \left[ \|  \widetilde{\vu}(s, \cdot) \|_{W^{3,2}(\Omega)}^2 + 1 \right] \textup{d}s \right) \int_{0}^{t} \left[\|  \widetilde{\vu}(s, \cdot) \|_{W^{3,2}(\Omega)}^2 + 1 \right]  \textup{d}s \right) \\
			&\leq \| \vr_0\|_{W^{3,2}(\Omega)}^2 \left(1+ C_1  \int_{0}^{T} \left[\|  \widetilde{\vu}(t, \cdot) \|_{W^{3,2}(\Omega)}^2 + 1 \right]  \dt \right) \exp \left(C_1\int_{0}^{T} \left[ \|  \widetilde{\vu}(t, \cdot) \|_{W^{3,2}(\Omega)}^2 + 1 \right] \dt \right). 
		\end{align*}
		Since the right-hand side is independent of $t \in [0,T]$, we can take the supremum on the left-hand side and obtain estimate \eqref{estimate density}. 
	\end{proof}

	Secondly, we focus on the solvability of the elliptic system. 
	
	\begin{lemma} \label{ES}
		For any fixed $\widetilde{\vr} \in C([0,T]; W^{3,2}(\Omega))$ and $\ff \in L^2(0,T; W^{3,2}(\Omega))$, there exists a unique $p \in L^2(0,T; W^{4,2}(\Omega))$ solving the elliptic system
		\begin{equation} \label{elliptic system}
			\begin{cases}
				\dive (A \Grad p) = \dive (\widetilde{\vr} A\vf ) &\mbox{in } (0,T) \times \Omega, \\
				(A \Grad p) \cdot \textbf{\textup{n}} = (\widetilde{\vr} A \vf) \cdot \textbf{\textup{n}} &\mbox{on } (0,T) \times \partial \Omega,
			\end{cases}
		\end{equation}
		in a distributional sense; in particular, the following estimate holds:
		\begin{align} \label{pressure estimate}
			\|p\|_{L^2(0,T; W^{4,2}(\Omega))} \lesssim \|\tilde{\vr}\|_{L^\infty(0,T; W^{3,2}(\Omega))} \|\ff\|_{L^2(0,T;W^{3,2}(\Omega))}.
		\end{align}
	\end{lemma}
	\begin{proof}
		To simplify the notation, we denote $g= \dive (\widetilde{\vr} A\vf )$ and $h= (\widetilde{\vr} A \vf) \cdot \textbf{\textup{n}}$. \\
		Multiplying the first equation of \eqref{elliptic system} by $\varphi \in  C^1([0,T] \times \overline{\Omega} )$ and integrating the resulting expression over $(0,T) \times \Omega$, through an integration by parts we get the identity
		\begin{equation*}
			\int_{0}^{T}\int_{\Omega} A \Grad p \cdot \Grad \varphi \ \dx\dt = - \int_{0}^{T}\int_{\Omega} g \varphi \ \dx \dt  + \int_{0}^{T} \int_{\partial \Omega} h\varphi \ \textup{d} S_x  \dt . 
		\end{equation*}
		Therefore, we fix the Hilbert space $H:= L^2(0,T; W^{4,2}(\Omega))$ and set
		\begin{align*}
			&a(p, \varphi) := \int_{0}^{T}\int_{\Omega} A \Grad p \cdot \Grad \varphi \ \dx \dt, \\
			&F(\varphi) := - \int_{0}^{T}\int_{\Omega} g \varphi \ \dx \dt + \int_{0}^{T} \int_{\partial \Omega} h\varphi \ \textup{d} S_x \dt . 
		\end{align*}
		Clearly, $a : H \times H \to \mathbb{R}$ is a continuous bilinear form, and, due to the fact that the matrix $A$ is a symmetric positive definite matrix, it is also coercive. Moreover, $F$ is a linear and continuous functional on $H$, hence $F \in H'$. \\
		As consequence of the Lax-Milgram theorem, there exists a unique $p \in H$ such that $a(p, \varphi)= F(\varphi)$ for any $\varphi \in H $ and such that $\| p\|_H \lesssim \| F\|_{H'}$, leading to estimate \eqref{pressure estimate}. More precisely, due to the embedding $W^{3,2}(\Omega) \hookrightarrow L^\infty(\Omega)$ and the fact that $W^{1,2}(\Omega) \hookrightarrow L^2(\d \Omega)$, we have
		\begin{align*}
		\|F\|_{H'} &\lesssim \|g\|_{L^2(0,T; L^2(\Omega))} + \|h\|_{L^2(0,T; L^2(\d \Omega))} \\
		&\lesssim \|\Grad (\tilde{\vr} \ff)\|_{L^2(0,T; L^2(\Omega))} + \|\tilde{\vr} \ff\|_{L^2(0,T; W^{1,2}(\Omega))} \\
		&\lesssim \|\tilde{\vr}\|_{L^\infty(0,T; W^{1,2}(\Omega))} \|\ff\|_{L^2(0,T; L^\infty(\Omega))} + \|\tilde{\vr}\|_{L^\infty(0,T; L^\infty(\Omega))} \|\Grad \ff\|_{L^2(0,T; L^2(\Omega))} \\
		&\lesssim \|\tilde{\vr}\|_{L^\infty(0,T; W^{3,2}(\Omega))} \|\ff\|_{L^2(0,T; W^{3,2}(\Omega))}.
		\end{align*}
	\end{proof}
	
	Finally, if we fix 
	\begin{equation*}
		\widetilde{\vr} \in C([0,T]; W^{3,2}(\Omega)), \quad \ff \in L^2(0,T; W^{3,2}(\Omega)), \quad \widetilde{p} \in L^2(0,T; W^{4,2}(\Omega)),
	\end{equation*}
	the corresponding $\vu$, uniquely determined by the second equation of \eqref{Darcy2}, clearly belongs to the class 
	\begin{equation*}
		\vu \in L^2(0,T; W^{3,2}(\Omega; \mathbb{R}^3)), 
	\end{equation*}
	with  
	\begin{align} \label{estimate velocity}
		\| \vu \|_{L^2(0,T; W^{3,2}(\Omega; \mathbb{R}^3))} &\lesssim \| \widetilde{\vr}\|_{L^\infty(0,T; W^{3,2}(\Omega))} \| \vf \|_{L^2(0,T; W^{3,2}(\Omega; \mathbb{R}^3))} + \| \widetilde{p}\|_{L^2(0,T; W^{4,2}(\Omega))} \notag \\
		&\lesssim \| \widetilde{\vr}\|_{L^\infty(0,T; W^{3,2}(\Omega))} \| \vf \|_{L^2(0,T; W^{3,2}(\Omega; \mathbb{R}^3))}.
	\end{align}

	We are now ready to prove the local existence result Theorem~\ref{Local Existence}. Given the initial datum $\vr_0 \in W^{3,2}(\Omega)$ and $B=2 \|\vr_0\|_{W^{3,2}(\Omega)}^2$, for some positive $T$ to be fixed later, we consider the space $X_T= C([0,T]; W^{2,2}(\Omega))$ and the set
	\begin{equation*}
		K_{T} := \left\{ \widetilde{\vr} \in X_T  \ : \quad \begin{aligned}
			&\quad \ \widetilde{\vr} \in C([0,T]; W^{3,2}(\Omega)), \quad \partial_t \widetilde{\vr} \in L^2(0,T; W^{2,2}(\Omega)), \\ 
			&\widetilde{\vr}(0, \cdot)=\vr_0, \quad  \sup_{t \in [0, T]} \| \widetilde{\vr}(t, \cdot) \|_{W^{3,2}(\Omega)}^2 + \int_{0}^{T} \| \partial_t \widetilde{\vr} \|_{W^{2,2}(\Omega)}^2 \dt \leq B
		\end{aligned}  \right\}.
	\end{equation*}
	Since $\tilde{\vr}(t,x) = \vr_0(x) \in K_{T}$, $K_T$ is a non-empty, convex and closed subset of $X_T$. Moreover, as consequence of Aubin-Lions lemma, $K_T$ is compact in $X_T$. 
	
	Now,  for any fixed $\widetilde{\vr} \in K_{T}$, by Lemma~\ref{ES} we find the corresponding pressure $p=p[\widetilde{\vr}]$, and consequently the velocity $\vu= \vu[\widetilde{\vr}, \widetilde{p}]= \vu[\widetilde{\vr}]$. Therefore, we define the map $L$ on $K_{T}$ such that $\vr :=L[\widetilde{\vr}]$ is the solution of the transport equation
	\begin{equation*}
		\partial_t \vr + \vu[\widetilde{\vr}] \cdot \Grad \vr  =0, \quad \vr(0,\cdot)=\vr_0. 
	\end{equation*}
	Combining estimates \eqref{estimate velocity} and \eqref{pressure estimate} with \eqref{estimate density}, we also recover from the last line that
	\begin{equation}\label{inequSolL}
			\begin{aligned}
				&\sup_{t \in [0, T]} \| L[\widetilde{\vr}] (t,\cdot) \|_{W^{3,2}(\Omega)}^2 + \int_{0}^{T} \| \partial_t L[\widetilde{\vr}] \|_{W^{2,2}(\Omega)}^2 \dt \\
				&\leq \sup_{t \in [0, T]} \| L[\widetilde{\vr}] (t,\cdot) \|_{W^{3,2}(\Omega)}^2 \left(1+ \int_{0}^{T} \| \vu[\widetilde{\vr}] \|_{W^{2,2}(\Omega)}^2 \dt \right) \\
				&\leq \| \vr_0\|_{W^{3,2}(\Omega)}^2 \left( 1+ C_2 B \int_{0}^{T} \| \vf \|_{W^{3,2}(\Omega; \mathbb{R}^3)}^2 \dt + C_2 T\right)^2 \exp \left( C_2 B \int_{0}^{T} \| \vf \|_{W^{3,2}(\Omega; \mathbb{R}^3)}^2 \dt + C_2 T\right).
			\end{aligned}
	\end{equation}
	Hence, choosing $T$ sufficiently small, say $T=T^*$, we get that $L(K_{T^*}) \subset K_{T^*}$.
	Furthermore, the map $L$ is continuous in $X_{T^*}$; indeed, if we consider a sequence $\{ \widetilde{\vr}_n \}_{n \in \mathbb{N}} \subset K_{T^*}$ such that $$ \widetilde{\vr}_n \to \widetilde{\vr} \quad \mbox{in } X_{T^*}=C([0,T^*]; W^{2,2}(\Omega)),$$ and we define $\vr_n:=L[\widetilde{\vr}_n]$,  $\vr :=L[\widetilde{\vr}]$, taking the difference of the equations satisfied by $\vr_n$ and $\vr$, multiplying it by $\vr_n-\vr$ and integrating the resulting expression over $\Omega$, we get the identity 
	\begin{equation} \label{e5}
			\frac12 \frac{\textup{d}}{\dt} \int_{\Omega} (\vr_n- \vr)^2 \dx + \int_{\Omega} (\vr_n- \vr) \big(\vu [\widetilde{\vr}_n]- \vu[\widetilde{\vr}]\big) \cdot \Grad \vr_n \dx =0,
	\end{equation}
	where we have used the fact that, through an integration by parts, 
	\begin{equation*}
			\int_{\Omega } (\vr_n- \vr)\  \vu[\widetilde{\vr}] \cdot \Grad(\vr_n- \vr) \ \dx =  \frac12 \int_{\Omega} \vu[\widetilde{\vr}] \cdot \Grad |\vr_n- \vr|^2 \ \dx =0. 
	\end{equation*}
	Integrating \eqref{e5} over $(0,\tau) \times \Omega$ with $0\leq \tau\leq T^*$, by Young's inequality we obtain the inequality
	\begin{equation*}
		\| (\vr_n- \vr) (\tau, \cdot) \|_{L^2(\Omega)}^2 \lesssim B \int_{0}^{\tau} \left\| \vu[\widetilde{\vr}_n] - \vu[\widetilde{\vr}] \right\|_{L^2(\Omega; \mathbb{R}^3)}^2 \dt + \int_{0}^{\tau} \|  \vr_n- \vr \|_{L^2(\Omega)}^2 \dt,
	\end{equation*}
	where we used that $\sup_{t \in [0, T^*]} \| \Grad \vr_n\|_{L^{\infty}(\Omega)}^2 \lesssim \sup_{t \in [0, T^*]} \|  \vr_n\|_{W^{3,2}(\Omega)}^2 \lesssim B$ since $\vr_n \in K_{T^*}$.	By a Gr\"onwall argument and \eqref{estimate velocity}, we therefore obtain that
	\begin{equation*}
		\sup_{t \in [0, T^*]} \| (\vr_n- \vr) (t, \cdot) \|_{L^2(\Omega)}^2 \leq C\left(\|\vf\|_{L^2(0,T; W^{3,2}(\Omega; \mathbb{R}^3))}, B, T^*\right) \sup_{t \in [0, T^*]} \| (\widetilde{\vr}_n - \widetilde{\vr}) (t, \cdot) \|_{L^2(\Omega)}^2. 
	\end{equation*}
	Hence, $\vr_n \to \vr$ in $C([0,T^*]; L^2(\Omega))$, which implies the convergence in $C([0,T^*]; W^{2,2}(\Omega))$ since $K_{T^*}$ is a compact subset of $X_{T^*}$. By Schauder's fixed point theorem, we get that $L$ has a fixed point $\vr \in K_{T^*}$, yielding the solution we were seeking for. 
	
	Finally, if $(\vr, \vu, p)$ is a solution of \eqref{Darcy2} in $(0,T^*) \times \Omega$ belonging to the regularity class 
	\begin{align*}
		&\vr \in C([0,T^*]; W^{3,2}(\Omega)), \quad \d_t \vr \in L^2(0,T^*; W^{2,2}(\Omega)), \\
		&\quad (\vu, p) \in L^2(0,T^*; W^{3,2}(\Omega; \mathbb{R}^3) \times W^{4,2}(\Omega) ),
	\end{align*} 
	then, taking the time-derivative of \eqref{elliptic system} and using elliptic estimates, it can be deduced that
	\begin{equation*}
		( \d_t \vu, \d_t p) \in L^2(0,T^*; W^{2,2}(\Omega; \mathbb{R}^3) \times W^{3,2}(\Omega) ), 
	\end{equation*}
	and that the solution is unique in this class, provided $\ff \in L^\infty(0, T^*; W^{3,2}(\Omega))$. This concludes the proof of Theorem~\ref{Local Existence}.
	
	\begin{remark}
	As a consequence of estimate \eqref{inequSolL}, we can find a global-in-time solution, provided the initial datum $\|\vr^0\|_{W^{3,2}(\O)} \ll 1$ such that still $L(K_T) \subset K_T$.
	\end{remark}
	
	\section*{Acknowledgements}
	{\it The work of D.B. is supported by the PRIN project 2022 ``Partial differential equations and related geometric-functional inequalities'', financially supported by the EU, in the framework of the ``Next Generation EU initiative''. The Department of Mathematics of Politecnico di Milano is supported by MUR ``Excellence Department 2023-2027''. F.O. has been supported by the Czech Academy of Sciences project L100192351. The Institute of Mathematics, CAS is supported by RVO:67985840. J.P. is supported  by the Natural Science Foundation of Jiangsu Province under grant BK20240058 and by the National Natural Science Foundation of China under Grant 12171235.}
	
	\section*{Conflict of interest}
	The authors declare no conflict of interest in this paper.
	
	
	\bibliographystyle{amsalpha}
	
	\newcommand{\etalchar}[1]{$^{#1}$}
\providecommand{\bysame}{\leavevmode\hbox to3em{\hrulefill}\thinspace}
\providecommand{\MR}{\relax\ifhmode\unskip\space\fi MR }
\providecommand{\MRhref}[2]{%
  \href{http://www.ams.org/mathscinet-getitem?mr=#1}{#2}
}
\providecommand{\href}[2]{#2}

\end{document}